\documentclass[11pt,letterpaper]{amsart}
\usepackage{dslihead}

\title{The plectic conjecture over local fields}

\author{Siyan Daniel Li-Huerta}
\address{Max-Planck-Institut f\"ur Mathematik\\Vivatsgasse 7\\53111 Bonn, Germany}
\email{siyan@mpim-bonn.mpg.de}

\swapnumbers
\theoremstyle{plain}

\newtheorem*{thm*}{Theorem}

\newtheorem*{lem*}{Lemma}
\newtheorem{prop}[subsection]{Proposition}
\newtheorem*{prop*}{Proposition}
\newtheorem*{conj*}{Conjecture}
\newtheorem{cor}[subsection]{Corollary}
\newtheorem*{cor*}{Corollary}

\newtheorem*{thmA}{Theorem A}
\newtheorem*{thmB}{Theorem B}
\newtheorem*{thmC}{Theorem C}
\newtheorem*{thmD}{Theorem D}

\theoremstyle{definition}
\newtheorem{defn}[subsection]{Definition}
\newtheorem*{defn*}{Definition}

\theoremstyle{remark}
\newtheorem{rem}[subsection]{Remark}
\newtheorem*{rem*}{Remark}
\newtheorem*{rems*}{Remarks}

\begin{document}

\begin{abstract}
Using a mixed-characteristic incarnation of fusion, we prove an analog of Nekov\'a\v{r}--Scholl's plectic conjecture for local Shimura varieties. We apply this to obtain results on the plectic conjecture for (global) Shimura varieties after restricting to a decomposition group. Along the way, we prove a $p$-adic uniformization theorem for the basic locus of abelian type Shimura varieties at hyperspecial level, which is of independent interest.
\end{abstract}

\dedicatory{In memory of Jan Nekov\'a\v{r}}

\maketitle
\tableofcontents

\section*{Introduction}

The \emph{plectic conjecture} was formulated by Nekov\'a\v{r}--Scholl \cite{NS16} as part of their program for constructing Euler systems beyond the rank-1 case. We begin by stating their conjecture for $\ell$-adic cohomology.\footnote{Nekov\'a\v{r}--Scholl \cite{NS16} observe that plectic phenomena seem to be motivic in nature. In particular, there are also archimedean \cite{NS17, Ai18, For24} and $p$-adic \cite{FGM22, FG23, LZ25} aspects, and they should all be related. However, we will say nothing about this.} Consider a Shimura datum $(G,X)$, and suppose that $G$ is the Weil restriction $\R_{F/\bQ}H$ of a connected reductive group $H$ over a totally real field $F$ of degree $d$. Then the action of the absolute Galois group $\Ga_\bQ$ on conjugacy classes of cocharacters of $G_{\ov\bQ}$ extends to an action of the \emph{plectic Galois group} $\Ga^{\plec}_{F/\bQ}\coloneqq\Aut_F(F\otimes_\bQ\ov\bQ)$. Write $\Ga^{[\mu]}_{F/\bQ}$ for the stabilizer of the inverse Hodge cocharacter $[\mu]$ in $\Ga^{\plec}_{F/\bQ}$, and note that the reflex field $E$ is characterized by $\Ga_E=\Ga_\bQ\cap\Ga^{[\mu]}_{F/\bQ}$. For any compact open subgroup $K$ of $G(\bA_f)$, write $\ov\Sh_K(G,X)$ for the minimal compactification of the Shimura variety $\Sh_K(G,X)$ at level $K$ over $E$.
\begin{conj*}[{\cite[Conjecture 6.1]{NS16}}\footnote{In \cite{NS16}, this conjecture was stated only for compact Shimura varieties, so it instead used $R\Ga(\Sh_K(G,X)_{\ov\bQ},\ov\bQ_\ell)$. For general Shimura varieties, $\IH(\ov\Sh_K(G,X)_{\ov\bQ},\ov\bQ_\ell)$ is a natural replacement, since the Kottwitz conjecture \cite[p.~201]{Kot90} predicts that the $\Ga_E$-action on its cohomology groups extends to a $\Ga^{[\mu]}_{F/\bQ}$-action. However, we note that \cite{DS20} proves that the $\Ga_\bQ$-action on the \'etale cohomology groups of \emph{open} Hilbert modular varieties extends to a $\Ga^{\plec}_{F/\bQ}$-action.}]
The complex $\IH(\ov\Sh_K(G,X)_{\ov\bQ},\ov\bQ_\ell)$ canonically lifts from an object of $D^b(\Ga_E,\ov\bQ_\ell)$ to an object of $D^b(\Ga^{[\mu]}_{F/\bQ},\ov\bQ_\ell)$.
\end{conj*}
Note that $\Ga^{\plec}_{F/\bQ}$ is a much larger profinite group than $\Ga_\bQ$. For example, when $F$ is Galois over $\bQ$, the restriction of $\Ga_\bQ\hra\Ga^{\plec}_{F/\bQ}\cong\Ga_F^d\rtimes\fS_d$ to $\Ga_F$ equals a conjugate of the diagonal embedding $\Ga_F\hra\Ga_F^d$.
\begin{rem*}
 We now outline the conjectural relation with rank-$d$ Euler systems. (For more details, see \cite[\S1.3]{mythesis}.) The plectic conjecture lets us take derived $\Ga^{[\mu]}_{F/\bQ}$-invariants of $\IH(\ov\Sh_K(G,X)_{\ov\bQ},\ov\bQ_\ell)$, which we apply to the case of CM points and to the case of Hilbert modular varieties. The plectic conjecture is known for CM points \cite[Proposition 6.4]{NS16}, and the resulting $\Ga^{[\mu]}_{F/\bQ}$-invariants can be analyzed explicitly using complex multiplication.

  We turn to Hilbert modular varieties. The second spectral sequence of hypercohomology, along with results of Brylinski--Labesse \cite{BL84} and Nekov\'a\v{r} \cite{Nek18}, shows that after localizing at a Hilbert cusp eigenform $f$, the derived $\Ga^{[\mu]}_{F/\bQ}$-invariants equal the $d$-th wedge power of the Selmer group associated with $f$. This receives a Gysin map from the $\Ga^{[\mu]}_{F/\bQ}$-invariants for CM points, and elements in the image of this map are expected to form a rank-$d$ Kolyvagin system in the sense of Mazur--Rubin \cite{MR16}. Moreover, the nonvanishing of this rank-$d$ Kolyvagin system is expected to be related to the $d$-th derivative of the $L$-function of $f$.
\end{rem*}

We proved a function field analog of the plectic conjecture in \cite{LH21}, which crucially used fusion in a moduli space of shtukas with $d$ legs. In light of this, the plectic conjecture can be viewed as a shadow of \emph{multiple-leg phenomena} over number fields. Given Yun--Zhang's formula \cite{YZ17} on interpreting the $d$-th derivative of an $L$-function in terms of cycles on precisely such a moduli space, this perspective helps explain the conjectural relation with rank-$d$ Euler systems outlined above. 

Currently, the plectic conjecture is wide open. The goal of this paper is to prove a version of the plectic conjecture after restricting to decomposition groups. For this, we use a mixed-characteristic incarnation of fusion in an essential way.

To state our results, we need some notation. Fix a prime $p\neq\ell$, and write $F_p$ for the \'etale $\bQ_p$-algebra $F\otimes_\bQ\bQ_p$. The choice of an embedding $\ov\bQ\hra\ov\bQ_p$ lets us realize the \emph{plectic decomposition group} $\Ga^{\plec}_{F_p/\bQ_p}\coloneqq\Aut_{F_p}(F\otimes_\bQ\ov\bQ_p)$ as a closed subgroup of $\Ga^{\plec}_{F/\bQ}$. Write $\Ga^{[\mu]}_{F_p/\bQ_p}$ for the stabilizer of $[\mu]$ in $\Ga^{\plec}_{F_p/\bQ_p}$, and note that the closure $E_v$ of $E$ in $\ov\bQ_p$ is characterized by $\Ga_{E_v}=\Ga_{\bQ_p}\cap\Ga^{\plec}_{F_p/\bQ_p}$.

Our first result concerns the basic locus $\Sh_K(G,X)_{\breve{E}_v}^b$ in the rigid analytification $\Sh_K(G,X)_{\breve{E}_v}^{\an}$ of $\Sh_K(G,X)$ over $\breve{E}_v$. Fix an isomorphism $\bC\ra^\sim\bC_p$.
\begin{thmA}
Suppose that $p\neq2$, $(G,X)$ is of abelian type, and that $G_{\bQ_p}$ is unramified. Let $K^p$ be a compact open subgroup of $G(\bA_f^p)$, and let $K_p$ be a hyperspecial compact open subgroup of $G(\bQ_p)$. Then the complex $R\Ga_c(\Sh_{K_pK^p}(G,X)_{\bC_p}^b,\ov\bQ_\ell)$ canonically lifts from an object of $D^b(\Ga_{E_v},\ov\bQ_\ell)$ to an object of $D^b(\Ga_{F_p/\bQ_p}^{[\mu]},\ov\bQ_\ell)$.
\end{thmA}
Our next result applies in certain ramified cases, when $\Sh_K(G,X)^{\an}_{\breve{E}_v}$ admits a $p$-adic uniformization by the Drinfeld tower. (For what this means, see \ref{ss:Drinfeldcase}. In particular, it implies that $\Sh_K(G,X)$ is projective over $E$.) Here, we obtain a result for all of $\Sh_K(G,X)$.
\begin{thmB}
Suppose that $\Sh_K(G,X)^{\an}_{\breve{E}_v}$ admits a $p$-adic uniformization by the Drinfeld tower as in \ref{ss:Drinfeldcase}. Let $K^p$ be a compact open subgroup of $G(\bA_f^p)$, and let $K_p$ be a maximal compact open subgroup of $G(\bQ_p)$. Then the complex $R\Ga(\Sh_K(G,X)_{\ov\bQ_p},\ov\bQ_\ell)$ canonically lifts from an object of $D^b(\Ga_{E_v},\ov\bQ_\ell)$ to an object of $D^b(\Ga_{F_p/\bQ_p}^{[\mu]},\ov\bQ_\ell)$.
\end{thmB}
We deduce Theorem A and Theorem B from the following analogous result for local Shimura varieties. Let $(G,b,\mu)$ be a local Shimura datum over $\bQ_p$, and suppose that $G$ is the Weil restriction $\R_{F_u/\bQ_p}H$ of a connected reductive group $H$ over a finite extension $F_u$ of $\bQ_p$. Form Weil group versions $W_{F_u/\bQ_p}^{[\mu]}\subseteq W_{F_u/\bQ_p}^{\plec}$ of the plectic Galois groups from before, and note that the reflex field $E_v$ is characterized by $W_{E_v}=W_{\bQ_p}\cap W^{[\mu]}_{F_u/\bQ_p}$. For any compact open subgroup $K$ of $G(\bQ_p)$, write $\breve\cM_{G,b,\mu,K}$ for the associated local Shimura variety at level $K$. Recall that $\breve\cM_{G,b,\mu,K}$ has an action of $J_b(\bQ_p)$.\footnote{In the body of the paper, we will use the notation $G_b$ instead of the usual $J_b$.}
\begin{thmC}
The complex $R\Ga_c(\breve\cM_{G,b,\mu,K,\bC_p},\ov\bQ_\ell)$ lifts canonically from an object of $D^b(J_b(\bQ_p)\times W_E,\ov\bQ_\ell)$ to an object of $D^b(J_b(\bQ_p)\times W^{[\mu]}_{F/\bQ_p},\ov\bQ_\ell)$.
\end{thmC}
\begin{rems*}\hfill
  \begin{enumerate}[(1)]
  \item Our methods apply more generally to the intersection homology complexes of \emph{moduli spaces of local shtukas} over any nonarchimedean local field. See Theorem \ref{ss:maintheorem}.
  \item Theorem A, Theorem B, and Theorem C also hold with integral coefficients. More precisely, we can replace $\ov\bQ_\ell$ with the ring of integers $\La$ of a finite extension of $\bQ_\ell$ containing $\sqrt{q}$. See Theorem \ref{ss:unramifiedabeliandecompositionplectic}, Theorem \ref{ss:Varshavskydecompositionplectic}, and Corollary \ref{ss:TheoremC}.
  \item A very similar approach was considered by Tamiozzo \cite{Tam19}. For example, he amazingly constructs a version of the plectic diagram (\ref{eq:plectic}) below for \emph{Hilbert modular surfaces} localized at a prime $p$ splitting in $F$, so that the twisted diagonal morphism (\ref{eq:twisteddiagonal}) below is replaced with the actual diagonal morphism $\Spd\bQ_p\ra(\Spd\bQ_p)^2$. While our plectic diagram involves local rather than global Shimura varieties, we can deduce the expected consequences for $\ell$-adic cohomology.
  \end{enumerate}
\end{rems*}
Let us sketch the proof of Theorem C. For simplicity, assume that $\mu$ is defined over $\bQ_p$, which implies that $W^{[\mu]}_{F_u/\bQ_p}=W^{\plec}_{F_u/\bQ_p}$. Write $d$ for $[F_u:\bQ_p]$. How can we canonically construct an action of $W^{\plec}_{F_u/\bQ_p}\cong W^d_{F_u}\rtimes\fS_d$ on the level of complexes? According to an adaption of Nekov\'a\v{r}--Scholl's vision \cite[(1.3)]{NS16} to local Shimura varieties, we want to find a Cartesian \emph{plectic diagram}
\begin{align*}
  \xymatrix{\breve\cM_{G,b,\mu,K}\ar[d]\ar@{-->}[r] & \text{``}\breve\cM_{G,b,\mu,K}^{\plec}\text{''}\ar@{-->}[d] \\
  \Spa\breve\bQ_p\ar@{-->}[r] & \text{``}(\Spa\breve\bQ_p)^{\plec}\text{''},
}
\end{align*}
where $\ov\bQ_\ell$-sheaves on $\text{``}(\Spa\breve\bQ_p)^{\plec}\text{''}$ yield representations of $W^{\plec}_{F_u/\bQ_p}$ over $\ov\bQ_\ell$.

After passing to Scholze's category of diamonds \cite{Sch17}, our proof begins along these lines. First, we use $\Spd\breve\bQ_p/\phi_{\bQ_p}$ instead of $\Spa\breve\bQ_p$, since this conveniently incorporates the Weil descent datum on $\breve{\cM}_{G,b,\mu,K}$, and also $\ov\bQ_\ell$-sheaves on $\Spd\breve\bQ_p/\phi_{\bQ_p}$ yield representations of $W_{\bQ_p}$ over $\ov\bQ_\ell$.

Next, we introduce a \emph{twisted diagonal morphism}
\begin{align}\label{eq:twisteddiagonal}
m^{-1}:\Spd\breve\bQ_p/\phi_{\bQ_p}\ra(\Spd\breve{F}_u/\phi_{F_u})^d/\fS_d,\tag{$\wr$}
\end{align}
as well as a \emph{symmetrized} version $\Sht^{(d)}_{H,b,\mu,K}\ra(\Spd\breve{F}_u/\phi_{F_u})^d/\fS_d$ of the moduli space of mixed-characteristic shtukas. Then, we use these to construct a Cartesian diagram
\begin{align}\label{eq:plectic}
\begin{split}
  \xymatrix{\Sht_{G,b,\mu,K}\ar[r]\ar[d] & \Sht^{(d)}_{H,b,\mu,K}\ar[d]\\
\Spd\breve\bQ_p/\phi_{\bQ_p}\ar[r]^-{m^{-1}} & (\Spd\breve{F}/\phi_F)^d/\fS_d,
  }
\end{split}\tag{$\Box$}
\end{align}
where $\Sht_{G,b,\mu,K}$ is the Weil descent of $\breve\cM_{G,b,\mu,K}$. This diagram says that $\Sht_{G,b,\mu,K}$ is obtained from restricting $\Sht^{(d)}_{H,b,\mu,K}$ along a (twisted) diagonal morphism, so we regard it as a mixed-characteristic incarnation of fusion.

We want to apply the plectic diagram to a ``relative $\IC$ sheaf'' on $\Sht^{(d)}_{H,b,\mu,K}$ over $(\Spd\breve{F}/\phi_F)^d/\fS_d$, except there is no general theory of $\IC$ sheaves for diamonds yet. To circumvent this, we define the desired sheaf by hand, and then we use results of Fargues--Scholze \cite{FS21} to prove that it is well-behaved.

Another problem is that not all $\ov\bQ_\ell$-sheaves on $(\Spd\breve{F}/\phi_F)^d/\fS_d$ yield representations of $W^{\plec}_{F/\bQ_p}$ over $\ov\bQ_\ell$. In order to prove that our sheaf (really, complex) does, we construct another Cartesian diagram
\begin{align*}
  \xymatrix{\Sht_{H,b,\mu,K}^{(d)}\ar[d] & \ar[l]\Sht^d_{H,b,\mu_\bullet,K}\ar[d]\\
(\Spd\breve{F}/\phi_F)^d/\fS_d &\ar[l](\Spd\breve{F}/\phi_F)^d.}
\end{align*}
Here, $\mu_\bullet$ is a $d$-tuple of conjugacy classes of cocharacters of $H_{\ov\bQ_p}$ derived from $\mu$, and $\Sht_{H,b,\mu_\bullet,K}^d\ra(\Spd\breve{F}/\phi_F)^d$ is the usual moduli space of mixed-characteristic shtukas with $d$ legs for $H$ over $F$. Using Fargues--Scholze's smoothness result for the intersection homology of $\Sht^d_{H,b,\mu_\bullet,K}$ \cite[Proposition IX.3.2]{FS21}, along with the fact that the image of $m^{-1}$ lies in the locus where $(\Spd\breve{F}/\phi_F)^d/\fS_d$ agrees with the stacky quotient $[(\Spd\breve{F}/\phi_F)^d/\fS_d]$, we conclude the proof of Theorem C.
\begin{rem*}
To prove Theorem A and Theorem B, we also need a generalization of Theorem C to finite \'etale $\bQ_p$-algebras $F$. See Remark \ref{rem:fetalgebras}.
\end{rem*}
With Theorem C in hand, let us return to the global context and sketch the proofs of Theorem A and Theorem B. Both proofs appeal to $p$-adic uniformization theorems for the basic locus of $\Sh_K(G,X)$. In the situation of Theorem B, the basic locus consists of all of $\Sh_K(G,X)$, and the relevant theorem is due to Rapoport--Zink \cite{RZ96} and Varshavsky \cite{Var98}. In the situation of Theorem A, we prove the following result.
\begin{thmD}
  Suppose that $p\neq2$ and $(G,X)$ is of abelian type. Let $K^p$ be a compact open subgroup of $G(\bA_f^p)$, and let $K_p$ be a hyperspecial compact open subgroup of $G(\bQ_p)$. We have an isomorphism
  \begin{align*}
    G'(\bQ)\bs(\breve\fM_{G,b,\mu_{\bC_p}}\times\ul{G(\bA_f^p)/K^p})\ra^\sim\fS\fh_{K_pK^p}(G,X)^b
  \end{align*}
  of formal schemes over $\breve\cO_{E_v}$ that is compatible with the Weil descent datum and varying $K^p$.
\end{thmD}
Here, $\breve\fM_{G,b,\mu_{\bC_p}}$ denotes the associated integral local Shimura variety over $\breve\cO_{E_v}$ as in \cite[Theorem 2.5.4]{PR22}, $\fS\fh_{K_pK^p}(G,X)^b$ denotes the completion of the integral canonical model $\sS_{K_pK^p}(G,X)_{\breve\cO_{E_v}}$ as in \cite[(3.4.14)]{Kis10} along the basic locus as in \cite[p.~183]{SZ17}, and $G'$ denotes the canonical inner form of $G$ over $\bQ$ associated with $(G,X)$ as in \cite[Proposition 3.1]{Han21}.
\begin{rem*}
  If $(G,X)$ is of Hodge type, then $G$ is never of the form $\R_{F/\bQ}H$ unless $F=\bQ$. To see this, note that the $\bQ$-split rank of $Z_G=\R_{F/\bQ}Z_H$ equals the $F$-split rank of $Z_H$, while the $\bR$-split rank of $Z_G$ equals the sum of the $\bR$-split ranks of $Z_H$ for all embeddings $F\ra\bR$. In particular, this is at least $d$ times the $F$-split rank of $Z_H$. But the $\bQ$- and $\bR$-split ranks of $Z_G$ coincide when $(G,X)$ is of Hodge type, which forces $d=1$.

Consequently, when considering the plectic conjecture as stated, we must work beyond the Hodge type case. Alternatively, one could instead formulate a modification of the plectic conjecture along the lines of \cite{Leo19}, which includes some Hodge type cases, although we do not pursue this here.
\end{rem*}

We prove Theorem D by developing a theory of connected components for basic $p$-adic uniformization at hyperspecial level. This involves interweaving work of Kisin \cite{Kis17} over the special fiber with results of Pappas--Rapoport \cite{PR22}. We then apply this theory to deduce Theorem D from Kim's results \cite{Kim18} in the Hodge type case, concluding the proof of Theorem A.

\subsection*{Outline}
In \S\ref{s:FFcurves}, we recall facts about the Fargues--Fontaine curve, as well as introduce the twisted diagonal morphism (\ref{eq:twisteddiagonal}). In \S\ref{s:symmetrizedSatakesheaves}, we construct symmetrized Satake sheaves and prove that they are well-behaved. In \S\ref{s:localshtukas}, we introduce symmetrized moduli space of local shtukas. In \S\ref{s:plectic}, we construct the plectic diagram (\ref{eq:plectic}) and prove Theorem C. In \S\ref{s:Shimuravarieties}, we apply our results to Shimura varieties and prove Theorem A and Theorem B. Finally, in Appendix \ref{s:appendix} we prove Theorem D.

\subsection*{Acknowledgements}
The author thanks Mark Kisin for his advice about Shimura varieties and their connected components. The author would also like to thank David Hansen and Michael Rapoport for helpful discussions, and to thank the referee for their comments and suggestions.

During the revision of this work, the author was a visitor of the Max Planck Institute for Mathematics. The author thanks MPIM for its hospitality and excellent working conditions.

\section{Fargues--Fontaine curves}\label{s:FFcurves}
In this section, we begin by recalling some facts about the (relative) Fargues--Fontaine curve. Next, we introduce a \emph{twisted diagonal morphism} (\ref{eq:twisteddiagonal}) for divisors on the Fargues--Fontaine curve, which is essential for our main results. Finally, we recall the moduli of bundles on the Fargues--Fontaine curve and describe how it behaves under Weil restriction.

We freely use definitions from perfectoid geometry as in \cite{Sch17} and \cite{FS21}. Unless otherwise specified, we work on the v-site of affinoid perfectoid spaces over $\ov\bF_q$. Note that this is equivalent to considering v-stacks over $\Spd\ov\bF_q$.

\subsection{}
Let $Q$ be a nonarchimedean local field, and fix an algebraic closure $\ov{Q}$ of $Q$. Choose a uniformizer $\pi$ of $Q$, and write $\bF_q$ for $\cO_Q/\pi$. For any affinoid perfectoid space $S=\Spa(R,R^+)$ over $\ov\bF_q$ with pseudouniformizer $\vpi$, write $Y_{S,Q}$ for the sousperfectoid adic space
\begin{align*}
\Spa W_{\cO_Q}(R^+)\ssm \{\pi[\vpi]=0\}
\end{align*}
\cite[Proposition II.1.1]{FS21}. We get a morphism $Y_{S,Q}\ra\Spa\breve{Q}$, where $\breve{Q}$ denotes $W_{\cO_Q}(\ov\bF_q)[\textstyle\frac1\pi]$. Fix an embedding $\breve{Q}\ra\wh{\ov{Q}}$ over $Q$. Write $\vp=\vp_Q$ for the automorphism of $Y_{S,Q}$ induced by the lift of $q$-Frobenius on $W_{\cO_Q}(R^+)$, and finally write $X_{S,Q}$ for the quotient $Y_{S,Q}/\vp^\bZ$, which has a morphism $X_{S,Q}\ra\Spa{Q}$.

\subsection{}\label{ss:relativeCartierdivisors}
Write $\phi:\Spd\breve{Q}\ra\Spd\breve{Q}$ for the geometric $q$-Frobenius automorphism over $\Spd\ov\bF_q$. Note that $\phi$ acts freely by \cite[Proposition II.1.16]{FS21} and \cite[Proposition II.1.18]{FS21}. In particular, the sheaf-theoretic quotient $\Spd\breve{Q}/\phi^\bZ$ coincides with the stack-theoretic quotient.

Write $W_Q$ for the absolute Weil group of $Q$ with respect to $\ov{Q}$. For any topological space $Z$, write $\ul{Z}$ for the associated v-sheaf as in \cite[p.~52]{Sch17}, and note that we have a morphism $\Spd\breve{Q}/\phi^\bZ\ra[*/\ul{W_Q}]$ corresponding to the $\ul{W_Q}$-bundle $\Spd{\wh{\ov{Q}}}\ra\Spd\breve{Q}/\phi^\bZ$.

Recall that $S$-points of $\Spd\breve{Q}/\phi^\bZ$ correspond bijectively to closed relative Cartier divisors of $X_{S,Q}$ with degree $1$, where ``closed'' refers to the associated ideals being closed subsets of the structure sheaf when evaluating on affinoids \cite[Proposition 5.3.8]{SW20}. Write $\Div^1_{Q}$ for $\Spd\breve{Q}/\phi^\bZ$, and more generally, write $\Div^d_Q$ for the sheaf-theoretic quotient $(\Div^1_{Q})^d/\fS_d$. Then $S$-points of $\Div^d_Q$ indeed correspond bijectively to closed relative Cartier divisors of $X_{S,Q}$ with degree $d$ \cite[Proposition II.3.6]{FS21}.

\subsection{}\label{ss:FFcurvefieldextension}
The following description of how the Fargues--Fontaine curve depends on the base field plays a crucial role in defining our twisted diagonal morphism, so we explain it in detail. Let $F$ be a degree $d$ finite separable extension of $Q$. Write $F_0$ for the maximal unramified subextension of $Q$ in $F$, and write $r$ for $[F_0:Q]$.
\begin{prop*}
  We have a natural Cartesian square
\begin{align*}
  \xymatrix{X_{S,F}\ar[r]^-m\ar[d] & X_{S,Q}\ar[d]\\
  \Spa{F}\ar[r] & \Spa Q}
\end{align*}
compatible with base change in $S$.
\end{prop*}
\begin{proof}
Note that we have Cartesian squares
\begin{align*}
  \xymatrix{\coprod_\io\Spa\breve{F}\ar[r]\ar[d] & \coprod_\io\Spa\breve{Q}\ar[r]\ar[d] & \Spa\breve{Q}\ar[d]\\
  \Spa{F}\ar[r] & \Spa{F_0}\ar[r] & \Spa{Q},}
\end{align*}
where $\io$ runs over the set $\Hom_Q(F_0,\breve{Q})$ of cardinality $r$. The base change of $\vp:\Spa\breve{Q}\ra\Spa\breve{Q}$ to $\coprod_\io\Spa\breve{Q}$ is induced by the automorphism $\prod_\io\breve{Q}\ra^\sim\textstyle\prod_\io\breve{Q}$ given by $(a_\io)_\io\mapsto(\vp(a_{\vp^{-1}\circ\io}))_\io$, and we have a similar description for the base change of $\vp$ to $\coprod_\io\Spa\breve{F}$. After further base changing to $Y_{S,Q}$, we get Cartesian squares
\begin{align*}
  \xymatrix{\coprod_\io Y_{S,F}\ar[r]\ar[d] & \coprod_\io Y_{S,Q}\ar[r]\ar[d] & Y_{S,Q}\ar[d]\\
  \Spa{F}\ar[r] & \Spa{F_0}\ar[r] & \Spa{Q}}
\end{align*}
along with an analogous description of the base change of $\vp:Y_{S,Q}\ra Y_{S,Q}$ to $\coprod_\io Y_{S,F}$. From this, we see that $\vp$ cyclically permutes the $r$ copies of $Y_{S,F}$, with $\vp^r$ acting on each copy as the automorphism induced by the lift of absolute $q^r$-Frobenius on $Y_{S,F}$. Hence quotienting the top row by $\vp^\bZ$ yields the desired Cartesian square.
\end{proof}

\subsection{}\label{ss:mapsfromfetcovers}
To define our twisted diagonal morphism, we will need the following v-stack version of \cite[Theorem (5.1)]{Ryd11}. Write $[d]$ for the finite set $\{1,\dotsc,d\}$.
\begin{lem*}
Let $Z$ be a v-stack. Then the prestack whose $S$-points parametrize finite \'etale morphisms $T\ra S$ of degree $d$ along with a morphism $T\ra Z$ is naturally isomorphic to the stack-theoretic quotient $[Z^d/\fS_d]$.
\end{lem*}
\begin{proof}
Any such $T\ra S$ is \'etale-locally isomorphic to $\ul{[d]}_S$ on the target. As the v-topology is subcanonical \cite[Theorem 8.7]{Sch17}, this yields an $\fS_d$-bundle $\ul{\Iso}(\ul{[d]},T)$ on $S$, and the morphism $T\ra Z$ induces an $\fS_d$-equivariant morphism $\ul{\Iso}(\ul{[d]},T)\ra Z^d$.

Conversely, let $M$ be an $\fS_d$-bundle on $S$ equipped with an $\fS_d$-equivariant morphism $M\ra Z^d$. Then $M\times^{\fS_d}\ul{[d]}$ is v-locally isomorphic to $\ul{[d]}_S$ over $S$, so it is finite \'etale of degree $d$ \cite[Proposition 9.7]{Sch17}. Postcomposing $M\ra Z^d$ with projections yields a morphism $M\times\ul{[d]}\ra Z$, and the $\fS_d$-equivariance of $M\ra Z^d$ indicates that this induces a morphism $M\times^{\fS_d}\ul{[d]}\ra Z$.
\end{proof}

\subsection{}\label{ss:divpullback}
Maintain the notation of \ref{ss:FFcurvefieldextension}. Write $(\Div^1_F)^{d,\circ}\subseteq(\Div^1_F)^d$ for the complement of all diagonals in $(\Div^1_F)^d$, which is an open subsheaf since $\Div^1_F$ is separated \cite[Proposition II.1.21]{FS21}. The symmetric group $\fS_d$ acts freely on $(\Div^1_F)^{d,\circ}$, so the sheaf-theoretic quotient $\Div^{d,\circ}_F\coloneqq(\Div^1_F)^{d,\circ}/\fS_d$ coincides with the stack-theoretic quotient. Finally, write $e$ for $[F:F_0]$.

We now define our twisted diagonal morphism for divisors, which is essential for our main results.
\begin{prop*}
  We have a natural closed embedding
  \begin{align*}
    m^{-1}:\Div^1_Q\ra[(\Div^1_F)^d/\fS_d].
  \end{align*}
  Its image lies in $\Div^{d,\circ}_F$, and the composition $\Div^1_Q\ra^{m^{-1}}[(\Div^1_F)^d/\fS_d]\ra\Div^d_F$ sends a closed relative Cartier divisor $D$ on $X_{S,Q}$ to its preimage $m^{-1}(D)$ under $m:X_{S,F}\ra X_{S,Q}$.
\end{prop*}
By abuse of notation, we also write $m^{-1}$ for $\Div^1_Q\ra[(\Div^1_F)^d/\fS_d]\ra\Div^d_F$.
\begin{proof}
By \cite[Lemma 15.6]{Sch17}, $\Spd\breve{F}\ra\Spd\breve{Q}$ is finite \'etale of degree $e$, so quotienting by appropriate powers of $\phi$ yields a finite \'etale morphism $m:\Div^1_F\ra\Div^1_Q$ of degree $d$. By Lemma \ref{ss:mapsfromfetcovers}, we get a morphism as desired via sending $S\ra\Div^1_Q$ to the data given by the Cartesian square
\begin{align*}
\xymatrix{T\ar[r]\ar[d] & \Div^1_F\ar[d]^-m \\
S\ar[r] & \Div^1_Q.}
\end{align*}
Unraveling the identifications in \ref{ss:relativeCartierdivisors} and the proof of Proposition \ref{ss:FFcurvefieldextension} yields our description of $\Div^1_Q\ra^{m^{-1}}[(\Div^1_F)^d/\fS_d]\ra\Div^d_F$. From here, the fact that $m$ is finite \'etale shows that the image of $m^{-1}$ lies in $\Div^{d,\circ}_F$. Therefore $m^{-1}$ is injective, and as its target and source are proper \cite[Proposition II.1.21]{FS21}, $m^{-1}$ is also proper and hence a closed embedding.
\end{proof}

\subsection{}\label{ss:BunG}
Let us recall some facts about bundles on the Fargues--Fontaine curve. Let $G$ be a connected reductive group over $Q$, and write $\Bun_G$ for the Artin v-stack over $\Spd\ov\bF_q$ whose $S$-points parametrize $G$-bundles on $X_{S,Q}$ \cite[Theorem IV.1.19]{FS21}.

For any $b$ in $B(G)$, write $G_b$ for the associated connected reductive group over $Q$ as in \cite[proposition (1.12)]{RZ96}\footnote{While \cite{RZ96} only treats $p$-adic $Q$, the proof immediately adapts to any $Q$. Also, we use the notation $G_b$ instead of the usual $J_b$ to emphasize its dependence on $G$.}, write $\cE_b$ for the associated $G$-bundle on $X_{S,Q}$ as in \cite[p.~89]{FS21}, and write $\wt{G}_b$ for the v-sheaf $\ul{\Aut}(\cE_b)$. We have a natural morphism $\ul{G_b(Q)}\ra\wt{G}_b$, which is an isomorphism if and only if $b$ is basic \cite[Proposition III.5.1]{FS21}.

The map $B(G)\ra|\Bun_G|$ given by $b\mapsto\cE_b$ is bijective \cite[Theorem III.2.2]{FS21}, and its inverse is continuous for the order topology on $B(G)$ by \cite[Theorem III.2.3]{FS21} and \cite[Theorem III.2.7]{FS21}. For any $b$ in $B(G)$, the resulting locally closed substack of $\Bun_G$ is naturally isomorphic to $[*/\wt{G}_b]$ \cite[Proposition III.5.3]{FS21}.

\begin{rem*}
The map $B(G)\ra|\Bun_G|$ is even a homeomorphism \cite[Theorem 1.1]{Vie21}, but we will not need this.
\end{rem*}

\subsection{}\label{ss:Weilrestrictionbundles}
Bundles satisfy the following compatibility with Weil restriction. Maintain the notation of \ref{ss:divpullback}. Let $H$ be a connected reductive group over $F$, and take $G$ to be the Weil restriction $\R_{F/Q}H$. Let $U$ be an adic space over $X_{S,Q}$, and assume that the fiber product $m^{-1}(U)\coloneqq U\times_{X_{S,Q}}X_{S,F}$ is also an adic space. Proposition \ref{ss:FFcurvefieldextension} implies that $G_U$ is naturally isomorphic to the Weil restriction $\R_{m^{-1}(U)/U}(H_U)$.
\begin{prop*}
The groupoid of $G$-bundles on $U$ is naturally equivalent to the groupoid of $H$-bundles on $m^{-1}(U)$.
\end{prop*}
\begin{proof}
  Let $\cG$ be a $G$-bundle on $U$, and write $\ve:G_{m^{-1}(U)}\ra H_{m^{-1}(U)}$ for the counit of the base change-Weil restriction adjunction. Because $\ve$ is a morphism of group adic spaces over $m^{-1}(U)$, we can form the $H$-bundle $\ve_*(\cG_{m^{-1}(U)})$ on $m^{-1}(U)$.

  Conversely, let $\cH$ be an $H$-bundle on $m^{-1}(U)$. Since Weil restrictions preserve finite products, we see that $\R_{m^{-1}(U)/U}\cH$ is a $G$-bundle on $U$.
\end{proof}

\begin{cor}\label{cor:geometrizeShapiro}
We have a natural isomorphism $c:\Bun_G\ra^\sim\Bun_H$ compatible with the bijection $B(G)\ra^\sim B(H)$ from \cite[1.10]{Kot85}\footnote{In \cite{Kot85}, Kottwitz only considers $p$-adic $Q$. However, everything works for general $Q$.}.
\end{cor}
\begin{proof}
Apply Proposition \ref{ss:Weilrestrictionbundles} to $U=X_{S,Q}$, using the fact that the equivalence commutes with base change in $S$. Proposition \ref{ss:Weilrestrictionbundles} geometrizes Shapiro's lemma, which is precisely how $B(G)\ra^\sim B(H)$ is constructed, so compatibility with this bijection follows.
\end{proof}

\section{Symmetrized Satake sheaves}\label{s:symmetrizedSatakesheaves}
The geometric Satake correspondence is an equivalence of categories between algebraic representations of the Langlands dual of $G$ and perverse sheaves on the \emph{local Hecke stack}. Under this equivalence, highest weight modules correspond to $\IC$ sheaves of affine Schubert varieties, so we will call the latter \emph{Satake sheaves}.

In their work, Fargues--Scholze prove a version of the geometric Satake correspondence in the context of diamonds. For our main results, we need to generalize their Satake sheaves to the symmetrized setting, but there is no general theory of $\IC$ sheaves yet for diamonds, so we must construct symmetrized Satake sheaves by hand. This is our goal for this section.

To this end, we introduce symmetrized local Hecke stacks and Beilinson--Drinfeld affine Grassmannians. (Conveniently, we also use them in \S\ref{s:localshtukas} to define symmetrized moduli spaces of local shtukas.) We then consider symmetrized affine Schubert varieties, whose basic properties we prove by bootstrapping from work of Scholze--Weinstein. Finally, we define symmetrized Satake sheaves and use work of Fargues--Scholze to prove that they are well-behaved.

\subsection{}
Let $D$ be a closed relative Cartier divisor $D$ on $X_{S,Q}$. When $D$ is affinoid, write $B^+_{\dR}(D)$ for the ring of global sections of the completion of $\cO_{X_{S,Q}}$ along $D$, and write $B_{\dR}(D)$ for its the punctured version as in \cite[p.~192]{FS21}. In general, there exists an open cover of $S$ on which $D$ is affinoid \cite[Proposition VI.1.2]{FS21}. Let $I$ be a finite set.
\begin{defn*}
  Write $\Hloc^{(d)I}_G$ for the small v-stack over $\Spd\ov\bF_q$ that is the v-stackification\footnote{By the above comment, it suffices to take the stackification in the analytic topology.} of the prestack sending $S$ to the data of
  \begin{enumerate}[i)]
  \item for all $i$ in $I$, a point $D_i$ of $(\Div_Q^d)(S)$ that is affinoid,
  \item two $G$-bundles $\cE$ and $\cE'$ on $\Spec B^+_{\dR}(\sum_{i\in I}D_i)$,
  \item an isomorphism $\al:\cE|_{\Spec B_{\dR}(\sum_{i\in I}D_i)}\ra^\sim\cE'|_{\Spec B_{\dR}(\sum_{i\in I}D_i)}$ of $G$-bundles.
  \end{enumerate}
Write $f:\Hloc^{(d)I}_G\ra(\Div^d_Q)^I$  for the morphism sending the above data to $(D_i)_{i\in I}$. When $d=1$ or $I$ is a singleton, we omit it from our notation.
\end{defn*}

\begin{defn}\label{defn:BDaffineGrassmannian}
  Write $\Gr^{(d)I}_G$ for the small v-sheaf over $\Spd\ov\bF_q$ whose $S$-points parametrize data consisting of
  \begin{enumerate}[i)]
  \item for all $i$ in $I$, a point $D_i$ of $(\Div_Q^d)(S)$,
  \item a $G$-bundle $\cE'$ on $X_{S,Q}$,
  \item an isomorphism $\al:\cE_1|_{X_{S,Q}\ssm\sum_{i\in I}D_i}\ra^\sim\cE'|_{X_{S,Q}\ssm\sum_{i\in I}D_i}$ of $G$-bundles that is meromorphic along $\sum_{i\in I}D_i$ as in \cite[Definition 5.3.5]{SW20}.
  \end{enumerate}
Write $f:\Gr^{(d)I}_G\ra(\Div^d_Q)^I$  for the morphism sending the above data to $(D_i)_{i\in I}$.
\end{defn}
By the Beauville--Lazslo theorem, Definition \ref{defn:BDaffineGrassmannian} agrees with the pullback of \cite[Definition VI.1.8]{FS21} via $(\Div^d_Q)^I\ra\Div^{d\cdot\#I}_Q$ \cite[p.~97]{FS21}. Under this identification, the natural v-cover $\Gr^{(d)I}_G\ra\Hloc^{(d)I}_G$ restricts ii) (respectively iii)) to the formal neighborhood (respectively punctured formal neighborhood) of $\sum_{i\in I}D_i$.

\subsection{}
The symmetrized spaces are related to the unsymmetrized ones as follows. Write $d\times I$ for $[d]\times I$. We have Cartesian squares
\begin{align*}
  \xymatrix{\Hloc^{(d)I}_{G}\ar[d]^-f &\ar[l]_-\Sg \Hloc^{d\times I}_{G}\ar[d]^-f\\
  (\Div^d_Q)^I & \ar[l]_-\Sg (\Div^1_Q)^{d\times I}
  }\mbox{ and }\xymatrix{\Gr^{(d)I}_{G}\ar[d]^-f &\ar[l]_-\Sg \Gr^{d\times I}_{G}\ar[d]^-f\\
  (\Div^d_Q)^I & \ar[l]_-\Sg (\Div^1_Q)^{d\times I},
  }
\end{align*}
where the $\Sg$ send $(x_{h,i})_{h\in[d],i\in I}$ to $(\sum_{h=1}^dx_{h,i})_{i\in I}$ and preserve all other data. Note that $\fS_d^I$ acts on the right-hand sides via permuting the $(x_{h,i})_{h\in[d],i\in I}$, and the $\Sg$ are invariant with respect to this action.

\subsection{}\label{ss:affineSchubertcells}
We consider a symmetrized version of affine Schubert varieties. Let $T$ be a maximal subtorus of $G$ over $Q$, and write $X_*^+(T)$ for the set of dominant cocharacters of $T_{\ov{Q}}$ with respect to a fixed Borel subgroup $B\subseteq G_{\ov{Q}}$ containing $T_{\ov{Q}}$. Identify $X_*^+(T)$ with the set of conjugacy classes of cocharacters of $G_{\ov{Q}}$.

Let $\de_1,\dotsc,\de_k$ be a partition of $d\times I$ refining the partition given by the $\big\{d\times\{i\}\big\}_{i\in I}$, let $\mu_1,\dotsc,\mu_k$ be in $X_*^+(T)$, and write $E_j$ for the field of definition of $\mu_j$. Write $\mu_\bullet=(\mu_{h,i})_{h\in[d],i\in I}$ for the element of $X_*^+(T)^{d\times I}$ defined by setting $\mu_{h,i}=\mu_j$ for all $(h,i)$ in $\de_j$.
\begin{defn*}\hfill
  \begin{enumerate}[a)]
  \item Write $\Hloc^{d\times I}_{G,\leq\mu_\bullet}$ for the substack of
    \begin{align*}
    \Hloc^{d\times I}_G\times_{(\Div_Q^1)^{d\times I}}\textstyle\prod_{j=1}^k(\Div^1_{E_j})^{\de_j}
    \end{align*}
    whose $S$-points consist of $((x_{h,i})_{h\in[d],i\in I},\cE,\cE',\al)$ such that, for all geometric points $\ov{s}$ of $S$ and $(h,i)$ in $d\times I$, the relative position of $\al_{\ov{s}}$ at $x_{h,i,\ov{s}}$ is bounded by $\sum_{(h',i')}\mu_{h',i'}$, where $(h',i')$ runs over elements of $d\times I$ that satisfy $x_{h',i',\ov{s}} = x_{h,i,\ov{s}}$. Write $\Hloc^{d\times I}_{G,\mu_\bullet}\subseteq\Hloc^{d\times I}_{G,\leq\mu_\bullet}$ for the substack where the relative position of $\al_{\ov{s}}$ at $x_{h,i,\ov{s}}$ equals $\sum_{(h',i')}\mu_{h',i'}$.
  \item Write $\Hloc^{(d)I}_{G,\leq\mu_\bullet}$ for the image of $\Hloc^{d\times I}_{G,\leq\mu_\bullet}$ in $\Hloc^{(d)I}_G\times_{(\Div^d_Q)^I}\prod_{j=1}^k\Div^{\#\de_j}_{E_j}$ under $\Sg$, and write $\Hloc^{(d)I}_{G,\mu_\bullet}\subseteq\Hloc^{(d)I}_{G,\leq\mu_\bullet}$ for the image of $\Hloc^{d\times I}_{G,\mu_\bullet}$.
  \end{enumerate}
  Write $\Gr^{(d)I}_{G,\leq\mu_\bullet}$ (respectively $\Gr^{(d)I}_{G,\mu_\bullet}$) for the preimage of $\Hloc^{(d)I}_{G,\leq\mu_\bullet}$ (respectively $\Hloc^{(d)I}_{G,\mu_\bullet}$) under
  \begin{align*}
    \Gr^{(d)I}_G\times_{(\Div^d_Q)^I}\textstyle\prod_{j=1}^k\Div^{\#\de_j}_{E_j}\ra\Hloc^{(d)I}_G\times_{(\Div^d_Q)^I}\textstyle\prod_{j=1}^k\Div^{\#\de_j}_{E_j}.
  \end{align*}
\end{defn*}

\subsection{}\label{lem:properdescends}
The following lemma enables us to check properness after passing to certain v-covers on the source.
\begin{lem*}
Let $f:Z'\ra Z$ be a morphism of v-stacks. Let $\wt{Z}\ra Z$ be a surjection that is representable in locally spatial diamonds, and write $\wt{f}:\wt{Z}'\ra\wt{Z}$ for the base change of $f$ to $\wt{Z}$. If $\wt{f}$ is proper, then $f$ is proper.
\end{lem*}
\begin{proof}
By \cite[Proposition 10.11]{Sch17}, we see that $f$ is quasicompact and separated. As for universally closed, let $S\ra Z$ be a morphism from a perfectoid space $S$, and consider $S'\coloneqq Z'\times_ZS\ra S$. Because $\wt{S}\coloneqq\wt{Z}\times_ZS$ is a locally spatial diamond with a surjective morphism $\wt{S}\ra S$, the map $|\wt{S}|\ra|S|$ is a quotient map by \cite[Proposition 11.18.(i)]{Sch17}, \cite[Lemma 12.11]{Sch17}, and \cite[Lemma 2.5]{Sch17}. The same holds for the base change $\wt{S}'\ra S'$ of $\wt{Z}'\ra Z$. Since $\wt{f}$ is universally closed, we see that $|\wt{S'}|\ra|\wt{S}|$ is closed. This implies that $|S'|\ra|S|$ is closed, as desired.
\end{proof}

\begin{prop}\label{prop:closedaffineSchubertcells}
The morphisms $\Hloc^{(d)I}_{G,\leq\mu_\bullet}\ra\Hloc^{(d)I}_G\times_{(\Div^d_Q)^I}\textstyle\prod_{j=1}^k\Div^{\#\de_j}_{E_j}$ and $\Gr^{(d)I}_{G,\leq\mu_\bullet}\ra\Gr^{(d)I}_G\times_{(\Div^d_Q)^I}\textstyle\prod_{j=1}^k\Div^{\#\de_j}_{E_j}$ are closed embeddings, and the morphism $f:\Gr^{(d)I}_{G,\leq\mu_\bullet}\ra\textstyle\prod_{j=1}^k\Div^{\#\de_j}_{E_j}$ is proper, representable in spatial diamonds, and of finite $\dimtrg$.
\end{prop}
\begin{proof}
  By applying \cite[Proposition 10.11.(i)]{Sch17} to the Cartesian square
  \begin{align*}
    \xymatrix{\Gr^{(d)I}_{G,\leq\mu_\bullet}\ar[r]\ar[d] & \Hloc^{(d)I}_{G,\leq\mu_\bullet}\ar[d]\\
    \Gr^{(d)I}_{G}\times_{(\Div^d_Q)^I}\textstyle\prod_{j=1}^k\Div^{\#\de_j}_{E_j}\ar[r] & \Hloc^{(d)I}_{G}\times_{(\Div^d_Q)^I}\textstyle\prod_{j=1}^k\Div^{\#\de_j}_{E_j},
    }
  \end{align*}
  we see that it suffices to prove the statements about $\Gr^{(d)I}_{G,\leq\mu_\bullet}$. Note that the $(\Div^1_{E_j})^{\de_j}\ra\Div^{\#\de_j}_{E_j}$ are quasi-pro-\'etale covers \cite[Proposition II.3.6]{FS21} and separated, so they are representable in locally spatial diamonds \cite[Proposition 13.6]{Sch17}. From here, the Cartesian square
  \begin{align*}
      \xymatrix{\Gr^{(d)I}_{G,\leq\mu_\bullet}\ar[d]^-{f} & \ar[l]_-\Sg\Gr^{d\times I}_{G,\leq\mu_\bullet}\ar[d]^-{f}\\
  \prod_{j=1}^k\Div^{\#\de_j}_{E_j} & \ar[l]_-\Sg \prod_{j=1}^k(\Div^1_{E_j})^{\de_j},}
  \end{align*}
Lemma \ref{lem:properdescends}, \cite[Proposition 10.11.(o)]{Sch17}, and \cite[Proposition 13.4.(iv)]{Sch17} reduce this to the analogous statements for $\Gr^{d\times I}_{G,\leq\mu_\bullet}$. The latter follow from \cite[Proposition 20.5.4]{SW20} by further pulling back to $\prod_{j=1}^k(\Spd\breve{E}_j)^{\de_j}$.
\end{proof}
\begin{cor}\label{cor:openaffineSchubertcells}
The morphisms $\Hloc^{(d)I}_{G,\mu_\bullet}\ra\Hloc^{(d)I}_{G,\leq\mu_\bullet}$ and $\Gr^{(d)I}_{G,\mu_\bullet}\ra\Gr^{(d)I}_{G,\leq\mu_\bullet}$ are open embeddings, and the morphism $f:\Gr^{(d)I}_{G,\mu_\bullet}\ra\textstyle\prod_{j=1}^k\Div^{\#\de_j}_{E_j}$ is compactifiable, representable in locally spatial diamonds, and of finite $\dimtrg$.
\end{cor}
\begin{proof}
This follows from Proposition \ref{prop:closedaffineSchubertcells} and the separatedness of $\Div_{E_j}^a$.
\end{proof}

\subsection{}\label{ss:symmetrizedSatakesheaves}
We now define Satake sheaves on our symmetrized local Hecke stacks. Write $j_{\mu_\bullet}:\Hloc^{(d)I}_{G,\mu_\bullet}\ra\Hloc^{(d)I}_G\times_{(\Div^d_Q)^I}\textstyle\prod_{j=1}^k\Div^{\#\de_j}_{E_j}$ for the locally closed embedding, write $2\rho$ for the sum of all positive roots of $G$, and set $d_{\mu_\bullet}\coloneqq\sum_{h\in[d],i\in I}\ang{2\rho,\mu_{h,i}}$. Let $\La$ be the ring of integers of a finite extension of $\bQ_\ell$ containing $\sqrt{q}$. Recall the perverse $t$-structure on
\begin{align*}
D_{\et}(\Hloc^{(d)I}_G\times_{(\Div^d_Q)^I}\textstyle\prod_{j=1}^k\Div^{\#\de_j}_{E_j},\Lambda)^{\mathrm{bd}},
\end{align*}
from \cite[Definition VI.7.1]{FS21}, where $(-)^{\mathrm{bd}}$ refers to being supported on a bounded substack as in \cite[Definition VI.2.6]{FS21}. Write $\IC_{\mu_\bullet}^{(d)I}$ for the complex $\prescript{p}{}\cH^0(j_{\mu_\bullet!}\La[d_{\mu_\bullet}])$.\footnote{While $\IC$ sheaves are usually defined as the image of $\prescript{p}{}\cH^0(j_!\La[d])\ra\prescript{p}{}\cH^0(j_*\La[d])$, for affine Schubert varieties this morphism is expected to be injective \cite[Proposition 8.2]{MV07}. For example, using \cite[Proposition VI.7.5]{FS21}, one can prove that it is an isomorphism after inverting $\ell$.}
\begin{prop*}
Our $\IC^{(d)I}_{\mu_\bullet}$ is flat perverse and universally locally acyclic over $\textstyle\prod_{j=1}^k\Div^{\#\de_j}_{E_j}$. Moreover, its pullback under $\Sg$ is naturally isomorphic to $\IC^{d\times I}_{\mu_\bullet}$, and the latter is isomorphic to the fusion product $\mathlarger{\mathlarger{\ast}}_{h\in[d],i\in I}\IC_{\mu_{h,i}}$.
\end{prop*}
\begin{proof}
By applying \cite[Proposition 22.19]{Sch17} and \cite[Proposition VI.7.4]{FS21} to the Cartesian square
\begin{align*}
      \xymatrix{\Hloc^{(d)I}_{G,\mu_\bullet}\ar[d]^-{j_{\mu_\bullet}} & \ar[l]_-\Sg\Hloc^{d\times I}_{G,\mu_\bullet}\ar[d]^-{j_{\mu_\bullet}}\\
  \Hloc^{(d)I}_G\times_{(\Div^d_Q)^I}\prod_{j=1}^k\Div_{E_j}^{\#\de_j} & \ar[l]_-\Sg \Hloc^{d\times I}_G\times_{(\Div^1_Q)^{d\times I}}\prod_{j=1}^k(\Div^1_{E_j})^{\de_j},}
\end{align*}
we see that $\Sg^*(\IC^{(d)I}_{\mu_\bullet}) = \IC^{d\times I}_{\mu_\bullet}$. By \cite[Corollary VI.7.6]{FS21} and \cite[Proposition IV.2.5]{FS21}, we can work v-locally on the base, so the Cartesian square
\begin{align*}
      \xymatrix{\Hloc^{(d)I}_{G,\mu_\bullet}\ar[d]^-{f} & \ar[l]_-\Sg\Hloc^{d\times I}_{G,\mu_\bullet}\ar[d]^-{f}\\
  \prod_{j=1}^k\Div^{\#\de_j}_{E_j} & \ar[l]_-\Sg \prod_{j=1}^k(\Div^1_{E_j})^{\de_j}}
\end{align*}
indicates that it suffices to prove the statements about $\IC^{d\times I}_{\mu_\bullet}$. After passing to a further v-cover, we may assume that $G$ is split over $Q$, so the $E_j=Q$.

Note that $j_{\mu_\bullet!}\La[d_{\mu_\bullet}]$ lies in $\prescript{p}{}D_{\et}^{\leq0}(\Hloc^{d\times I}_G,\La)^{\mathrm{bd}}$. Using the notation and results of \cite[Proposition VI.7.4]{FS21}, this implies that the shifted constant term $\CT_B(j_{\mu_\bullet!}\La[d_{\mu_\bullet}])[\deg]$ lies in $D_{\et}^{\leq0}(\Gr_T^{d\times I},\La)$. By specializing to the Witt affine Grassmannian and comparing with Mirkovi\'c--Vilonen cycles as in the proof of \cite[Proposition VI.7.5]{FS21}, we see that
\begin{align*}
  \CT_B(\IC^{d\times I}_{\mu_\bullet})[\deg] = \cH^0(\CT_B(j_{\mu_\bullet!}\La[d_{\mu_\bullet}])[\deg])
\end{align*}
is locally finite free over $\La$, where the equality holds via \cite[Proposition VI.7.4]{FS21}. Thus \cite[Proposition VI.6.4]{FS21} shows that $\IC^{d\times I}_{\mu_\bullet}$ is universally locally acyclic over $(\Div^1_Q)^I$, and \cite[Proposition VI.7.7]{FS21} shows it is flat perverse.

To identify $\IC^{d\times I}_{\mu_\bullet}$ with the fusion product, \cite[Proposition VI.9.3]{FS21} indicates that it suffices to restrict to $(\Div^1_Q)^{d\times I,\circ}$. Here the proof of \cite[Proposition VI.9.4]{FS21} shows that the fusion product equals the restriction of the exterior tensor product, so the result follows from the K\"unneth formula.
\end{proof}

\section{Symmetrized moduli spaces of local shtukas}\label{s:localshtukas}
To form the plectic diagram in \S\ref{s:plectic}, we need to generalize moduli spaces of local shtukas to the symmetrized setting. We begin by defining them and proving their basic properties. Next, we prove a smoothness theorem for their homology by reducing to the unsymmetrized special case, where it is a result of Fargues--Scholze.

At this point, we turn to the special case of local Shimura varieties. We recall Scholze--Weinstein's comparison with Rapoport--Zink spaces in the (P)EL case, and we conclude by proving an analogous comparison with Rapoport--Zink spaces at hyperspecial level in the unramified Hodge type case. We use this comparison in Appendix \ref{s:appendix} to prove $p$-adic uniformization results for global Shimura varieties in the unramified abelian type case.

\subsection{}
Let $b$ be in $B(G)$. We start by defining the unbounded version.
\begin{defn*}
  Write $\Sht_{G,b}^{(d)I}$ for the small v-sheaf over $\Spd\ov\bF_q$ whose $S$-points parametrize data consisting of
  \begin{enumerate}[i)]
  \item for all $i$ in $I$, a point $D_i$ of $(\Div^d_Q)(S)$,
  \item an isomorphism $\al:\cE_1|_{X_{S,Q}\ssm\sum_{i\in I}D_i}\ra^\sim\cE_b|_{X_{S,Q}\ssm\sum_{i\in I} D_i}$ of $G$-bundles that is meromorphic along $\sum_{i\in I}D_i$.
  \end{enumerate}
  In other words, $\Sht^{(d)I}_{G,b}$ is defined by the Cartesian square
  \begin{align*}
    \xymatrix{\Sht^{(d)I}_{G,b}\ar[r]\ar[d] & \Gr^{(d)I}_G\ar[d] \\
    \ast \ar[r]^-{\cE_b} & \Bun_G,}
  \end{align*}
where the right arrow sends $((D_i)_{i\in I},\cE',\al)$ to $\cE'$. Hence $\Sht^{(d)I}_{G,b}$ has commuting actions of $\ul{G(Q)}$ and $\ul{G_b(Q)}$ via \ref{ss:BunG} and composition with $\al$. Write $f:\Sht^{(d)I}_{G,b}\ra(\Div^d_Q)^I$ for the morphism sending the above data to $(D_i)_{i\in I}$. 
\end{defn*}

\subsection{}
Next, we use Definition \ref{ss:affineSchubertcells} to define the bounded version.
\begin{defn*}
Write $\Sht^{(d)I}_{G,b,\leq\mu_\bullet}$ for the preimage of $\Gr^{(d)I}_{G,\leq\mu_\bullet}$ under
\begin{align*}
\Sht^{(d)I}_{G,b}\times_{(\Div^d_Q)^I}\textstyle\prod_{j=1}^k\Div^{\#\de_j}_{E_j}\ra\Gr^{(d)I}_G\times_{(\Div^d_Q)^I}\textstyle\prod_{j=1}^k\Div^{\#\de_j}_{E_j}.
\end{align*}
Note that $\Sht^{(d)I}_{G,b,\leq\mu_\bullet}$ is preserved under the actions of $\ul{G(Q)}$ and $\ul{G_b(Q)}$. For any compact open subgroup $K$ of $G(Q)$, write $\Sht^{(d)I}_{G,b,\leq\mu_\bullet,K}$ for the quotient $\Sht^{(d)I}_{G,b,\leq\mu_\bullet}/\ul{K}$, and write $f_K:\Sht^{(d)I}_{G,b,\leq\mu_\bullet,K}\ra\textstyle\prod_{j=1}^k\Div^{\#\de_j}_{E_j}$ for the induced morphism.
\end{defn*}
By Proposition \ref{prop:closedaffineSchubertcells}, the morphism
\begin{align*}
\Sht^{(d)I}_{G,b,\leq\mu_\bullet}\ra\Sht^{(d)I}_{G,b}\times_{(\Div^d_Q)^I}\textstyle\prod_{j=1}^k\Div^{\#\de_j}_{E_j}
\end{align*}
is a closed embedding.

\begin{rem*}
We consider $\Sht^{(d)I}_{G,b,\leq\mu_\bullet,K}$ over $\textstyle\prod_{j=1}^k\Div^{\#\de_j}_{E_j}$ instead of pulling back to $\textstyle\prod_{j=1}^k(\Spd\breve{E}_j)^{\de_j}$ because it naturally incorporates partial Frobenius morphisms. For example, this yields a cleaner plectic diagram in Proposition \ref{ss:shtukaplecticdiagram}.
\end{rem*}

\begin{prop}
The morphism $f_K:\Sht^{(d)I}_{G,b,\leq\mu_\bullet,K}\ra\textstyle\prod_{j=1}^k\Div^{\#\de_j}_{E_j}$ is representable in locally spatial diamonds.
\end{prop}
\begin{proof}
Note that an $S$-point of $\Sht^{(d)I}_{G,b}/\ul{K}$ consists of
\begin{enumerate}[i)]
\item for all $i$ in $I$, a point $D_i$ of $\Div^d_Q(S)$,
\item a $G$-bundle $\cE$ on $X_{S,Q}$ such that, for all geometric points $\ov{s}$ of $S$, the fiber $\cE_{\ov{s}}$ is trivial,
\item an isomorphism $\al:\cE|_{X_{S,Q}\ssm\sum_{i\in I}D_i}\ra^\sim\cE_b|_{X_{S,Q}\ssm\sum_{i\in I}D_i}$ of $G$-bundles that is meromorphic along $\sum_{i\in I}D_i$,
\item a $\ul{K}$-bundle $\cP$ on $S$ whose pushforward along $\ul{K}\ra\ul{G(Q)}$ equals the $\ul{G(Q)}$-bundle on $S$ corresponding to $\cE$ via \ref{ss:BunG}.
\end{enumerate}
Because $[*/\ul{G(Q)}]\ra*$ and $[*/\ul{K}]\ra[*/\ul{G(Q)}]$ are separated by \cite[Proposition 10.11.(ii)]{Sch17}, we see that $f_K$ is separated. As the $(\Div^1_{E_j})^{\de_j}\ra\Div^{\#\de_j}_{E_j}$ are quasi-pro-\'etale covers \cite[Proposition II.3.6]{FS21}, the Cartesian square
\begin{align*}
  \xymatrix{\Sht^{(d)I}_{G,b,\leq\mu_\bullet,K}\ar[d]^-{f_K} & \ar[l]_-\Sg\Sht^{d\times I}_{G,b,\leq\mu_\bullet,K}\ar[d]^-{f_K}\\
  \prod_{j=1}^k\Div^{\#\de_k}_{E_j} & \ar[l]_-\Sg \prod_{j=1}^k(\Div^1_{E_j})^{\de_j}}
\end{align*}
and \cite[Proposition 13.4.(iv)]{Sch17} show that it suffices to prove that the right arrow is representable in locally spatial diamonds. But this follows from \cite[Theorem 23.1.4]{SW20}\footnote{While \cite{SW20} only discusses the $Q=\bQ_p$ case, the proofs adapt to any $Q$. Indeed, this is already implicitly used in \cite[Proposition XI.3.2]{FS21}.}, upon further applying \cite[Proposition 13.4.(iv)]{Sch17}.
\end{proof}

\subsection{}\label{ss:homologylocalsystem}
Recall the following result of Fargues--Scholze for the homology of $\Sht^I_{G,b,\leq\mu_\bullet,K}$. Write $\prescript{\prime}{}{\IC}_{\mu_\bullet}^{(d)I}$ for the object of $D_{\mathsmaller{\blacksquare}}(\Sht^{(d)I}_{G,b,\leq\mu_\bullet,K},\La)$ obtained from $\IC_{\mu_\bullet}^{(d)I}$ by first applying the double-dual embedding as in \cite[p.~264]{FS21} and then pulling back to $\Sht^{(d)I}_{G,b,\leq\mu_\bullet,K}$.
\begin{thm*}
  The object $f_{K\natural}(\prescript{\prime}{}{\IC}_{\mu_\bullet}^I)$ of $D_{\mathsmaller{\blacksquare}}([*/\ul{G_b(Q)}]\times\prod_{i\in I}\Div^1_{E_i},\La)$ arises via pullback from
  \begin{align*}
    D_{\mathrm{lis}}([*/\ul{G_b(Q)}],\La)^{B\prod_{i\in I}W_{E_i}}\subseteq D_{\mathsmaller{\blacksquare}}([*/\ul{G_b(Q)}]\times[*/\ul{\textstyle\prod_{i\in I}W_{E_i}}],\La).
  \end{align*}
\end{thm*}
\begin{proof}
  Recall that $[*/\ul{G_b(Q)}]$ arises from quotienting $f_K$ by $\ul{G_b(Q)}$. View the $\mu_i$ as dominant coweights of $\wh{G}$, and recall that their associated Weyl modules $V_{\mu_i}$ naturally have a $W_{E_i}$-action compatible with the $W_{E_i}$-action on $\wh{G}$. Hence \cite[Theorem VI.11.1]{FS21} associates with $V\coloneqq\mathlarger{\mathlarger{\boxtimes}}_{i\in I}V_{\mu_i}$ a flat perverse sheaf on $\Hloc_G^I\times_{(\Div^1_Q)^I}\prod_{i\in I}\Div^1_{E_i}$ that is universally locally acyclic over $\prod_{i\in I}\Div^1_{E_i}$. Proposition \ref{ss:symmetrizedSatakesheaves} indicates that this complex is precisely $\IC_{\mu_\bullet}$. Thus the desired result is \cite[Proposition IX.3.2]{FS21}, except we opt to work over $\prod_{i\in I}\Div^1_{E_i}$ rather than pulling back to $\prod_{i\in I}\Spd\breve{E}_i$.
\end{proof}

\subsection{}\label{ss:descendedhomology}
We descend Theorem \ref{ss:homologylocalsystem} to the symmetrized setting as follows. Note that $\textstyle\prod_{j=1}^k\fS_{\de_j}$ acts on $\Sht^{d\times I}_{G,b,\leq\mu_\bullet,K}$ via permuting the $(x_{h,i})_{h\in[d],i\in I}$, and $f_K:\Sht^{d\times I}_{G,b,\leq\mu_\bullet,K}\ra\textstyle\prod_{j=1}^k(\Div^1_{E_j})^{\de_j}$ is equivariant with respect to this action. Hence $f_{K\natural}(\prescript{\prime}{}{\IC}^{d\times I}_{\mu_\bullet})$ descends to an object $\mathrm{H}^{(d)I}_{\mu_\bullet}$ of $D_{\mathsmaller{\blacksquare}}([*/\ul{G_b(F)}]\times\textstyle\prod_{j=1}^k[(\Div^1_{E_j})^{\de_j}/\fS_{\de_j}],\La)$ by \cite[Proposition 17.3]{Sch17}.
\begin{prop*}
  The object $\mathrm{H}^{(d)I}_{\mu_\bullet}$ arises via pullback from
  \begin{align*}
    D_{\mathrm{lis}}([*/\ul{G_b(F)}],\La)^{B\prod_{j=1}^kW^{\de_j}_{E_j}\rtimes\fS_{\de_j}}\subseteq D_{\mathsmaller{\blacksquare}}([*/\ul{G_b(F)}]\times[*/\ul{\textstyle\prod_{j=1}^kW^{\de_j}_{E_j}\rtimes\fS_{\de_j}}],\La).
  \end{align*}
Moreover, the restriction of $\mathrm{H}^{(d)I}_{\mu_\bullet}$ to $[*/\ul{G_b(F)}]\times\textstyle\prod_{j=1}^k\Div^{\#\de_j,\circ}_{E_j}$ is naturally isomorphic to the restriction of $f_{K\natural}(\prescript{\prime}{}{\IC}^{(d)I}_{\mu_\bullet})$.
\end{prop*}
\begin{proof}
  The first claim follows from Theorem \ref{ss:homologylocalsystem} and \cite[Proposition 17.3]{Sch17}, since we have a commutative square
  \begin{align*}
    \xymatrix{\prod_{j=1}^k[(\Div^1_{E_j})^{\de_j}/\fS_{\de_j}]\ar[d] & \ar[l]_-\Sg \prod_{j=1}^k(\Div^1_{E_j})^{\de_j}\ar[d]\\
    [*/\ul{\prod_{j=1}^kW_{E_j}^{\de_j}\rtimes\fS_{\de_j}}] & \ar[l]_-\Sg [*/\ul{\prod_{j=1}^kW_{E_j}^{\de_j}}].
    }
  \end{align*}
  To prove the second claim, it suffices to see that the pullback under $\Sg$ of the restriction of $f_{K\natural}(\prescript{\prime}{}{\IC}^{(d)I}_{\mu_\bullet})$ to $\textstyle\prod_{j=1}^k\Div^{\#\de_j,\circ}_{E_j}$ is naturally isomorphic to the restriction of $f_{K\natural}(\prescript{\prime}{}{\IC}^{d\times I}_{\mu_\bullet})$ to $\prod_{j=1}^k(\Div^1_{E_j})^{\de_j,\circ}$. Corollary \ref{ss:symmetrizedSatakesheaves}, \cite[Proposition IV.2.15]{FS21}, and \cite[Proposition VII.4.1]{FS21} indicate that $\Sg^*(\prescript{\prime}{}{\IC}^{(d)I}_{\mu_\bullet})=\prescript{\prime}{}{\IC}^{d\times I}_{\mu_\bullet}$. So the result follows from the Cartesian square
  \begin{align*}
    \xymatrix{\Sht^{(d)I}_{G,b,\leq\mu_\bullet,K}\ar[d]^-{f_K} & \ar[l]_-\Sg \Sht^{d\times I}_{G,b,\leq\mu_\bullet,K}\ar[d]^-{f_K}\\
    \prod_{j=1}^k\Div^{\#\de_j}_{E_j} & \ar[l]_-\Sg \prod_{j=1}^k(\Div^1_{E_j})^{\de_j}
    }
  \end{align*}
combined with \cite[Proposition VII.3.1.(iii)]{FS21}.
\end{proof}

\subsection{}\label{ss:localShimuravarieties}
Let us recall the special case of local Shimura varieties. In this subsection, take $Q=\bQ_p$, $d=1$, and $I$ to be a singleton. Suppose that $(G,b,\mu)$ is a local Shimura datum as in \cite[Definition 5.1]{RV14}. Then there exists a unique smooth rigid space $\breve\cM_{G,b,\mu,K}$ over $\breve{E}$ for which $\breve\cM^{\diamondsuit}_{G,b,\mu,K}$ is isomorphic to $\Sht_{G,b,\leq\mu,K}\times_{\Div^1_E}\Spd\breve{E}$ over $\Spd\breve{E}$ \cite[p.~225]{SW20}. Note that the \'etale period morphism implies that $\breve\cM_{G,b,\mu,K}$ has dimension $d_\mu$.

By applying \cite[Proposition 10.2.3]{SW20} to the isomorphism
\begin{align*}
\Sht_{G,b,\leq\mu,K}\times_{\Div^1_E}\Spd\breve{E}\ra^\sim\vp^*_E(\Sht_{G,b,\leq\mu,K}\times_{\Div^1_E}\Spd\breve{E})
\end{align*}
arising via pulling back from $\Div^1_E$, we get a Weil descent datum on $\breve\cM_{G,b,\mu,K}$ over $\breve{E}$. We similarly obtain an action of $G_b(\bQ_p)$ on $\breve\cM_{G,b,\mu,K}$ over $\breve{E}$.

Write $\breve\cM_{G,b,\mu}$ for the locally spatial diamond $\Sht_{G,b,\leq\mu}\times_{\Div^1_E}\Spd\breve{E}$.

\subsection{}\label{ss:RapoportZinkspaces}
Maintain the notation of \ref{ss:localShimuravarieties}. EL type and PEL type Rapoport--Zink spaces can be viewed as local Shimura varieties as follows. In the EL case, let $\fD=(B,\cO_B,V,b,\mu,\cL)$ be an EL datum relative to $\ov\bF_q$ in the sense of \cite[definition (3.18)]{RZ96}. In the PEL case, where $p\neq2$, additionally include $(\cdot,\cdot)$ and $*$ as in \cite[definition (3.18)]{RZ96}. Suppose that $G$ is the associated connected reductive group over $\bQ_p$ as in \cite[(1.38)]{RZ96}, and for any compact open subgroup $K$ of $G(\bQ_p)$, write $\breve{\mathrm{M}}_{\fD,K}$ for the associated smooth rigid space over $\breve{E}$ as in \cite[(5.34)]{RZ96}. Recall that $\breve{\mathrm{M}}_{\fD,K}$ enjoys a Weil descent datum over $\breve{E}$ \cite[(5.47)]{RZ96} as well as an action of $G_b(\bQ_p)$ \cite[(5.35)]{RZ96}.
\begin{thm*}
  We have a natural isomorphism $\breve{\mathrm{M}}_{\fD,K}\ra^\sim\breve\cM_{G,b,\mu,K}$ that is compatible with the Weil descent datum, the $G_b(\bQ_p)$-action, and varying $K$.
\end{thm*}
\begin{proof}
This follows from \cite[Corollary 24.3.5]{SW20} and its proof.
\end{proof}

\subsection{}\label{ss:Kimspaces}
Maintain the notation of \ref{ss:localShimuravarieties}. Unramified Hodge type Rapoport--Zink spaces can also be viewed as local Shimura varieties as follows. Assume that $p\neq2$, adopt the notation of \ref{ss:integrallocalshimuravarieties}, and suppose that $(G,b,\mu)$ is of Hodge type as in \cite[Definition 2.5.5]{Kim18}. Write $\mathtt{RZ}_{G,b}$ for the associated formally smooth formal scheme over $\breve\cO_E$ as in \cite[Theorem 4.9.1]{Kim18}. Recall that $\mathtt{RZ}_{G,b}$ enjoys a Weil descent datum over $\breve\cO_E$ \cite[Definition 7.3.5]{Kim18} as well as an action of $G_b(\bQ_p)$ \cite[Section 7.2]{Kim18}.
\begin{thm*}
We have a natural isomorphism $\mathtt{RZ}_{G,b}\ra^\sim\breve\fM_{G,b,\mu}$ that is compatible with the Weil descent datum and the $G_b(\bQ_p)$-action.
\end{thm*}
\begin{proof}
  By \cite[Proposition 18.4.1]{SW20}, it suffices to construct this isomorphism between the associated diamonds. Let $\La$ be the faithful algebraic representation of $G$ over $\bZ_p$ as in \cite[Definition 2.5.5]{Kim18}. By definition, the corresponding closed embedding $G\ra\GL_h$ of group schemes over $\bZ_p$ induces a morphism $(G,b,\mu)\ra(\GL_h,b',\mu')$ of local Shimura data, where the latter arises from an EL datum $(\bQ_p,\bZ_p,\La\otimes_{\bZ_p}\bQ_p,b',\mu',\{p^\bZ\La\})$ relative to $\ov\bF_p$ as in \ref{ss:RapoportZinkspaces}.

  Note that \cite[Theorem 25.1.2]{SW20} and \cite[Theorem 4.9.1]{Kim18} give us a closed embedding $\mathtt{RZ}_{G,b}\ra\breve\fM_{\GL_h,b',\mu'}$, and \cite[Proposition 3.6.2]{PR21} gives us another closed embedding $\breve\fM_{G,b,\mu}\ra\breve\fM_{\GL_h,b',\mu'}$. Applying \cite[Lemma 12.5]{Sch17} to the inclusions of $\breve\fM_{G,b,\mu}\cap\mathtt{RZ}_{G,b}^{\diamondsuit}$ into $\breve\fM_{G,b,\mu}$ and $\mathtt{RZ}_{G,b}^{\diamondsuit}$ enables us to check the desired result on geometric points, where it follows immediately from \cite[Theorem 17.5.2]{SW20} and \cite[Theorem 9.10]{An22}.
\end{proof}
\begin{rem*}
  We could similarly obtain a natural isomorphism between the smooth rigid space $\mathtt{RZ}^K_{G,b}$ over $\breve{E}$ as in \cite[p.~91]{Kim18} and $\breve\cM_{G,b,\mu,K}$ that is compatible with the Weil descent datum, the $G_b(\bQ_p)$-action, and varying $K$. However, we will not need this.
\end{rem*}

\section{The plectic conjecture over local fields}\label{s:plectic}
Our goal for this section is to prove Theorem C. First, we apply the twisted diagonal morphism from \S\ref{s:FFcurves} to construct the plectic diagram for local Hecke stacks. For this, we elaborate on an observation of Nekov\'a\v{r}--Scholl to generalize the twisted diagonal morphism to account for \emph{plectic reflex fields}. We then prove that the symmetrized Satake sheaves from \S\ref{s:symmetrizedSatakesheaves} are compatible with the plectic diagram.

From here, we deduce the plectic diagram for moduli spaces of local shtukas, as well as the plectic conjecture over local fields. Finally, we specialize to the case of local Shimura varieties, proving Theorem C. We also explain how to extend our results to general finite \'etale $Q$-algebras $F$, which we use in \S\ref{s:Shimuravarieties} to prove our results on global Shimura varieties.

\subsection{}\label{ss:Hlocplecticdiagram}
We begin with the local Hecke stack incarnation of the plectic diagram. Recall the notation of Proposition \ref{ss:Weilrestrictionbundles}.
\begin{prop*}
  We have a natural Cartesian square
  \begin{align*}
    \xymatrix{\Hloc^I_G\ar[r]^-{(m^{-1})^I}\ar[d]^-f & \Hloc^{(d)I}_H\ar[d]^-f\\
    \prod_{i\in I}\Div_Q^1\ar[r]^-{(m^{-1})^I} & \prod_{i\in I}\Div_F^d.}
  \end{align*}
\end{prop*}
\begin{proof}
  Recall that an $S$-point of $\Hloc^I_G$ consists of
  \begin{enumerate}[i)]
  \item for all $i$ in $I$, a point $x_i$ of $(\Div^1_Q)(S)$,
  \item two $G$-bundles $\cE$ and $\cE'$ on $\Spec B^+_{\dR}(\sum_{i\in I}x_i)$,
  \item an isomorphism $\al:\cE_{\Spec B_{\dR}(\sum_{i\in I}x_i)}\ra^\sim\cE'|_{\Spec B_{\dR}(\sum_{i\in I}x_i)}$ of $G$-bundles.
  \end{enumerate}
  By applying Proposition \ref{ss:Weilrestrictionbundles} to $\Spec B_{\dR}^+(\sum_{i\in I}x_i)$ and $\Spec B_{\dR}(\sum_{i\in I}x_i)$ and using Proposition \ref{ss:FFcurvefieldextension}, we see that ii) and iii) are equivalent to two $H$-bundles $\cF$ and $\cF'$ on $\Spec B_{\dR}^+(\sum_{i\in I}m^{-1}(x_i))$ along with an isomorphism
  \begin{align*}
    \be:\cF|_{\Spec B_{\dR}(\sum_{i\in I}m^{-1}(x_i))}\ra^\sim\cF'|_{\Spec B_{\dR}(\sum_{i\in I}m^{-1}(x_i))}
  \end{align*}
  of $H$-bundles. Together with i), this is precisely the data parametrized by the above fiber product.
\end{proof}

\subsection{}\label{ss:plecticreflexgroup}
To proceed, we need the following elaboration of \cite[Proposition 5.2]{NS16} on \emph{plectic reflex groups}. Let $S$ be a maximal subtorus of $H$ over $F$, and take $T$ to be the Weil restriction $\R_{F/Q}S$. Then the $W_Q$-set $X_*^+(T)$ is induced from the $W_F$-set $X_*^+(S)$, so the $W_Q$-action on $X_*^+(T)$ naturally extends to a $W_{F/Q}^{\plec}$-action via the canonical map $W_Q\ra W_{F/Q}^{\plec}$ \cite[(4.1.5)]{Leo19}. For any $\mu$ in $X_*^+(T)$, write $W^{[\mu]}_{F/Q}$ for the stabilizer of $\mu$ in $W^{\plec}_{F/Q}$. 
\begin{prop*}
There exists an enumeration of representatives for $W_Q/W_F$ in $W_Q$ such that, under the resulting identification $W^{\plec}_{F/Q}\cong W_F^d\rtimes\fS_d$, the subgroup $W^{[\mu]}_{F/Q}$ corresponds to $\prod_{j=1}^kW_{F_j}^{\de_j}\rtimes\fS_{\de_j}$ for some partition $\de_1,\dotsc,\de_k$ of $[d]$ and finite separable extensions $F_1,\dotsc,F_k$ of $F$.
\end{prop*}

\begin{proof}
Any enumeration of representatives for $W_Q/W_F$ in $W_Q$ yields an identification $X_*^+(T)\cong X_*^+(S)^d$ such that the action of $W^{\plec}_{F/Q}\cong W_F^d\rtimes\fS_d$ is the natural one. View $\mu=(\mu_h)_{h\in[d]}$ as an element of $X_*^+(S)^d$, and choose our enumeration of representatives such that, if $\mu_h$ and $\mu_{h'}$ lie in the same $W_F$-orbit, then $\mu_h=\mu_{h'}$. Partitioning $[d]$ according to these $W_F$-orbits and letting $W_{F_j}$ be their corresponding stabilizers in $W_F$ yields the desired result.
\end{proof}

\subsection{}\label{ss:divpullbackreflex}
Maintain the notation of Proposition \ref{ss:plecticreflexgroup}, and write $E$ for the field of definition of $\mu$. Let us introduce a version of the twisted diagonal morphism that encapsulates the relation between reflex fields and plectic reflex fields.
\begin{prop*}
  We have a natural commutative square
  \begin{align*}
    \xymatrix{\Div^1_E\ar[r]^-{m^{-1}} \ar[d] & \prod_{j=1}^k[(\Div^1_{F_j})^{\de_j}/\fS_{\de_j}]\ar[d] \\
    \Div^1_Q\ar[r]^-{m^{-1}} & [(\Div^1_F)^d/\fS_d]
    }
  \end{align*}
that realizes $\Div^1_E$ as an open and closed subspace of the fiber product. Moreover, we also obtain a commutative square with the same property after replacing the right arrow with $\prod_{j=1}^k\Div^{\#\de_j}_{F_j}\ra\Div^d_F$.
\end{prop*}
\begin{proof}
  The proof of \cite[Theorem 16.3.1]{SW20} shows that the finite \'etale fundamental group of $(\Div^1_F)^d$ is $\Ga_F^d$. Thus the finite \'etale fundamental group of $[(\Div^1_F)^d/\fS_d]$ is $\Ga_F^d\rtimes\fS_d$, and its finite \'etale cover
  \begin{align*}
    \textstyle\prod_{j=1}^k[(\Div^1_{F_j})^{\de_j}/\fS_{\de_j}]\ra[(\Div^1_F)^d/\fS_d]
  \end{align*}
  corresponds to the open subgroup $\prod_{j=1}^k\Ga_{F_j}^{\de_j}\rtimes\fS_{\de_j}$.

  Using our enumeration of representatives for $W_Q/W_F$ in $W_Q$ from Proposition \ref{ss:plecticreflexgroup}, we see that the map on fundamental groups induced by
  \begin{align*}
    m^{-1}:\Div^1_Q\ra[(\Div^1_F)^d/\fS_d]
  \end{align*}
  is precisely $\Ga_Q\ra\Ga^{\plec}_{F/Q}\cong\Ga_F^d\rtimes\fS_d$. Note that $W_E$ equals the intersection of $W^{[\mu]}_{F/Q}$ with $W_Q$, so Proposition \ref{ss:plecticreflexgroup} indicates that the pullback of
  \begin{align*}
    \textstyle\prod_{j=1}^k[(\Div^1_{F_j})^{\de_j}/\fS_{\de_j}]\ra[(\Div^1_F)^d/\fS_d]
  \end{align*}
  to $\Div_Q^1$ has a connected component corresponding to the open subgroup $\Ga_E$ of $\Ga_Q$ (using the fact that absolute Galois and Weil groups agree on profinite completions). The latter corresponds to the finite \'etale cover $\Div^1_E\ra\Div^1_Q$, as desired.

By Proposition \ref{ss:divpullback}, the image of $m^{-1}$ lies in $\Div^{d,\circ}_F$. The preimage of $\Div^{d,\circ}_F$ in $\prod_{j=1}^k[(\Div^1_{F_j})^{\de_j}/\fS_{\de_j}]$ lies in $\prod_{j=1}^k\Div^{\#\de_j,\circ}_{F_j}$, so \ref{ss:divpullback} implies the last statement.
\end{proof}

\subsection{}\label{ss:affineSchubertcellsplecticdiagram}
The plectic diagram in Proposition \ref{ss:Hlocplecticdiagram} restricts to affine Schubert varieties as follows. Let $\mu_\bullet=(\mu_i)_{i\in I}$ be an element of $X_*^+(T)^I$. For any $i$ in $I$, write $E_i$ for the field of definition of $\mu_i$. Applying Proposition \ref{ss:plecticreflexgroup} to $\mu_i$ yields a partition $\de_{i,1},\dotsc,\de_{i,k_i}$ of $[d] = d\times\{i\}$ and finite separable extensions $F_{i,1},\dotsc,F_{i,k_i}$ of $F$, and the proof of Proposition \ref{ss:plecticreflexgroup} also yields $\mu_{i,1},\dotsc,\mu_{i,k_i}$ in $X_*^+(S)$ such that $F_{i,j}$ is the field of definition of $\mu_{i,j}$. Applying Proposition \ref{ss:divpullbackreflex} yields a natural morphism
\begin{align*}
m^{-1}:\Div^1_{E_i}\ra\textstyle\prod_{j=1}^{k_i}\Div^{\#\de_j}_{F_{i,j}}.
\end{align*}
As we vary $i$ over $I$, we obtain a partition $\coprod_{i\in I}\{\de_{i,1},\dotsc,\de_{i,k_i}\}$ of $d\times I$ indexed by $\coprod_{i\in I}[k_i]$, elements $\mu_{i,j}$ of $X_*^+(S)$ indexed by $\coprod_{i\in I}[k_i]$, and finite separable extensions $F_{i,j}$ of $F$ indexed by $\coprod_{i\in I}[k_i]$.

Let us relabel $\coprod_{i\in I}[k_i]$ as $[k]$, where $k\coloneqq\sum_{i\in I}k_i$. Then we get a partition $\de_1,\dotsc,\de_k$ of $d\times I$ refining the partition given by $\big\{d\times\{i\}\big\}_{i\in I}$, elements $\mu_1,\dotsc,\mu_k$ of $X_*^+(S)$, and finite separable extensions $F_1,\dotsc,F_k$ of $F$. By abuse of notation, we also write $\mu_\bullet$ for the associated element of $X_*^+(S)^{d\times I}$ as in \ref{ss:affineSchubertcells}, i.e. the $(h,i)$-th entry of $\mu_\bullet$ equals $\mu_j$ for all $(h,i)$ in $\de_j\subseteq d\times I$.
\begin{prop*}
  We have a natural Cartesian square
  \begin{align*}
    \xymatrix{\Hloc^I_{G,\leq\mu_\bullet}\ar[r]^-{(m^{-1})^I}\ar[d]^-f & \Hloc^{(d)I}_{H,\leq\mu_\bullet}\ar[d]^-f\\
    \prod_{i\in I}\Div^1_{E_i}\ar[r]^-{(m^{-1})^I} & \prod_{j=1}^k\Div^{\#\de_j}_{F_j}.}
  \end{align*}
Moreover, we also obtain a Cartesian square after replacing $\Hloc^I_{G,\leq\mu_\bullet}$ and $\Hloc^{(d)I}_{H,\leq\mu_\bullet}$ with $\Hloc^I_{G,\mu_\bullet}$ and $\Hloc^{(d)I}_{H,\mu_\bullet}$, respectively.
\end{prop*}
\begin{proof}
  Pulling back Proposition \ref{ss:Hlocplecticdiagram} to the top row of Proposition \ref{ss:divpullbackreflex} gives a Cartesian square
  \begin{align*}
    \xymatrix{\Hloc^I_{G}\times_{(\Div^1_Q)^I}\prod_{i\in I}\Div^1_{E_i}\ar[r]^-{(m^{-1})^I}\ar[d]^-f & \Hloc^{(d)I}_{H}\times_{(\Div^d_F)^I}\prod_{j=1}^k\Div^{\#\de_j}_{F_j}\ar[d]^-f\\
    \prod_{i\in I}\Div^1_{E_i}\ar[r]^-{(m^{-1})^I} & \prod_{j=1}^k\Div^{\#\de_j}_{F_j}.}
  \end{align*}
From here, the desired results follow by checking on geometric points.
\end{proof}

\subsection{}\label{ss:plecticSatakesheaves}
Recall the notation of Proposition \ref{ss:symmetrizedSatakesheaves}.
\begin{cor*}
The pullback of $\IC^{(d)I}_{\mu_\bullet}$ under $(m^{-1})^I$ is naturally isomorphic to $\IC^I_{\mu_\bullet}$.
\end{cor*}
\begin{proof}
Proposition \ref{ss:affineSchubertcellsplecticdiagram} and \cite[Proposition VI.7.4]{FS21} indicate that the $\prescript{p}{}\cH^0$ is preserved. The Cartesian square
  \begin{align*}
    \xymatrix{\Hloc^I_{G,\mu_\bullet}\ar[r]^-{(m^{-1})^I}\ar[d]^-{j_{\mu_\bullet}} & \Hloc^{(d)I}_{H,\mu_\bullet}\ar[d]^-{j_{\mu_\bullet}}\\
    \Hloc^I_{G}\times_{(\Div^1_Q)^I}\prod_{i\in I}\Div^1_{E_i}\ar[r]^-{(m^{-1})^I} & \Hloc^{(d)I}_{H}\times_{(\Div^d_F)^I}\prod_{j=1}^k\Div^{\#\de_j}_{F_j}}
  \end{align*}
and \cite[Proposition 22.19]{Sch17} indicate that the $j_{\mu_\bullet,!}$ is preserved.
\end{proof}

\subsection{}\label{ss:shtukaplecticdiagram}
Maintain the notation of Proposition \ref{ss:affineSchubertcellsplecticdiagram}, view $b$ as an element of $B(H)$ via the bijection $B(G)\ra^\sim B(H)$ from \cite[1.10]{Kot85}, and view $K$ as a compact open subgroup of $G(Q)=H(F)$. We arrive at the plectic diagram for moduli spaces of local shtukas.
\begin{prop*}
  We have a natural Cartesian square
  \begin{align*}
    \xymatrix{\Sht^I_{G,b,\leq\mu_\bullet,K}\ar[r]^-{(m^{-1})^I}\ar[d]^-{f_K} & \Sht^{(d)I}_{H,b,\leq\mu_\bullet,K}\ar[d]^-{f_K} \\
    \prod_{i\in I}\Div^1_{E_i}\ar[r]^-{(m^{-1})^I} & \prod_{j=1}^k\Div^{\#\de_j}_{F_j}.
    }
  \end{align*}
\end{prop*}
\begin{proof}
  Note that an $S$-point of $\Sht^I_{G,b,\leq\mu_\bullet}$ consists of
  \begin{enumerate}[i)]
  \item for all $i$ in $I$, a point $D_i$ of $(\Div^1_{E_i})(S)$,
  \item an isomorphism $\al:\cE_1|_{X_{S,Q}\ssm\sum_{i\in I}D_i}\ra^\sim\cE_b|_{X_{S,Q}\ssm\sum_{i\in I}D_i}$ of $G$-bundles that is meromorphic along $\sum_{i\in I}D_i$ and satisfies the relative position bound from Definition \ref{ss:affineSchubertcells}.a).
  \end{enumerate}
  By applying Proposition \ref{ss:Weilrestrictionbundles} to $X_{S,Q}\ssm\sum_{i\in I}D_i$, using Corollary \ref{cor:geometrizeShapiro} to identify the bundles, and using Proposition \ref{ss:affineSchubertcellsplecticdiagram} to compare relative position bounds, we see that ii) is equivalent to an isomorphism
  \begin{align*}
    \be:\cE_1|_{X_{S,F}\ssm\sum_{i\in I}m^{-1}(D_i)}\ra^\sim\cE_b|_{X_{S,F}\ssm\sum_{i\in I}m^{-1}(D_i)}
  \end{align*}
  of $H$-bundles that is meromorphic along $\sum_{i\in I}m^{-1}(D_i)$ and satisfies the relative position bound from Definition \ref{ss:affineSchubertcells}.b). Therefore we obtain a Cartesian square
  \begin{align*}
    \xymatrix{\Sht^I_{G,b,\leq\mu_\bullet}\ar[r]^-{(m^{-1})^I}\ar[d]^-{f} & \Sht^{(d)I}_{H,b,\leq\mu_\bullet}\ar[d]^-{f} \\
    \prod_{i\in I}\Div^1_{E_i}\ar[r]^-{(m^{-1})^I} & \prod_{j=1}^k\Div^{\#\de_j}_{F_j}.
    }    
  \end{align*}
Quotienting the top row by $\ul{K}$ yields the desired Cartesian square.
\end{proof}

\subsection{}\label{ss:maintheorem}
We now prove the plectic conjecture over local fields.
\begin{thm*}
The object $f_{K\natural}(\prescript{\prime}{}{\IC}_{\mu_\bullet}^I)$ of $D_{\mathrm{lis}}([*/\ul{G_b(Q)}],\La)^{B\prod_{i\in I}W_{E_i}}$ canonically lifts to an object of $D_{\mathrm{lis}}([*/\ul{G_b(Q)}],\La)^{B\prod_{j=1}^kW_{F_j}^{\de_j}\rtimes\fS_{\de_j}}$.
\end{thm*}
\begin{proof}
  First, I claim that the pullback of $f_{K\natural}(\prescript{\prime}{}{\IC}_{\mu_\bullet}^{(d)I})$ under $(m^{-1})^I$ is naturally isomorphic to $f_{K\natural}(\prescript{\prime}{}{\IC}_{\mu_\bullet}^I)$. By applying Corollary \ref{ss:plecticSatakesheaves}, \cite[Proposition IV.2.15]{FS21}, and \cite[Proposition VII.4.1]{FS21}, we see that $(m^{-1})^{I*}(\prescript{\prime}{}{\IC}^{(d)I}_{\mu_\bullet})=\prescript{\prime}{}{\IC}^I_{\mu_\bullet}$. Hence the claim follows from combining Proposition \ref{ss:shtukaplecticdiagram} with \cite[Proposition VII.3.1.(iii)]{FS21}.

By modifying the claim using Proposition \ref{ss:divpullback} and Proposition \ref{ss:descendedhomology}, we also see that the pullback of $\mathrm{H}^{(d)I}_{\mu_\bullet}$ under $(m^{-1})^I$ is naturally isomorphic to $f_{K\natural}(\prescript{\prime}{}{\IC}^I_{\mu_\bullet})$. Proposition \ref{ss:descendedhomology} also indicates that $\mathrm{H}^{(d)I}_{\mu_\bullet}$ arises via pullback from $D_{\mathrm{lis}}([*/\ul{H_b(F)}],\La)^{B\prod_{j=1}^kW_{F_j}^{\de_j}\rtimes\fS_{\de_j}}$, so the identification $G_b(Q)=H_b(F)$ and the commutative square
  \begin{align*}
    \xymatrix{\prod_{i\in I}\Div^1_{E_i}\ar[r]^-{(m^{-1})^I}\ar[d] & \prod_{j=1}^k[(\Div^1_{F_j})^{\de_j}/\fS_{\de_j}]\ar[d] \\
    [*/\ul{\prod_{i\in I}W_{E_i}}]\ar[r] & [*/\ul{\prod_{j=1}^kW_{F_j}^{\de_j}\rtimes\fS_{\de_j}}]
    }
  \end{align*}
enable us to conclude.
\end{proof}

\subsection{}\label{ss:TheoremC}
Let us specialize Theorem \ref{ss:maintheorem} to the situation of \ref{ss:localShimuravarieties}, proving Theorem C. Write $f_K$ for the structure morphism $\breve\cM_{G,b,\mu,K}\ra\Spa\breve{E}$, and recall that $Rf_{K!}\La$ is naturally an object of $D_{\mathrm{lis}}([*/\ul{G_b(\bQ_p)}],\La)^{BW_E}$ \cite[Theorem IX.3.1]{FS21}.
\begin{cor*}
  The object $Rf_{K!}\La$ canonically lifts to an object of
  \begin{align*}
    D_{\mathrm{lis}}([*/\ul{G_b(\bQ_p)}],\La)^{B\prod_{j=1}^kW_{F_j}^{\de_j}\rtimes\fS_{\de_j}}.
  \end{align*}
  \begin{proof}
    By \cite[(7.5.3)]{Hub96}, $\La$ is isomorphic to $Rf_K^!\La(-d_\mu)[-2d_\mu]$. Therefore $Rf_{K!}\La$ is isomorphic to $Rf_{K!}Rf_K^!\La(-d_\mu)[-2d_\mu]$. As in the proof of \cite[Theorem IX.3.1]{FS21}, we can identify $Rf_{K!}Rf_K^!\La$ with $f_{K\natural}(\prescript{\prime}{}{\IC}_\mu)$, which canonically lifts to an object of $D_{\mathrm{lis}}([*/\ul{G_b(\bQ_p)}],\La)^{B\prod_{j=1}^kW_{F_j}^{\de_j}\rtimes\fS_{\de_j}}$ by Theorem \ref{ss:maintheorem}. Furthermore, because $d_\mu=\sum_{j=1}^kd_{\mu_j}\cdot\#\de_j$, we explicitly see that $\La(-d_\mu)$ also canonically lifts to an object of $D_{\mathrm{lis}}([*/\ul{G_b(\bQ_p)}],\La)^{B\prod_{j=1}^kW_{F_j}^{\de_j}\rtimes\fS_{\de_j}}$. Hence taking shifts and tensor products yields the desired result.
  \end{proof}
\end{cor*}

\subsection{}\label{ss:Kunneth}
To extend our results to finite \'etale $Q$-algebras $F$, we need the K\"unneth formula for relative homology. This does not seem to be written down yet in the 5-functor formalism of \cite[Chapter VII]{FS21}, so we record it here. In this subsection, $\La$ is any solid $\wh\bZ^p$-algebra.
\begin{lem*}
Let $f:X\ra S$ and $g:Y\ra S$ be morphisms of small v-stacks, and let $A$ and $B$ be objects of $D_{\mathsmaller{\blacksquare}}(X,\La)$ and $D_{\mathsmaller{\blacksquare}}(Y,\La)$, respectively. Write $h:X\times_SY\ra S$ for the structure morphism. Then we have a natural isomorphism $h_\natural(\pr_1^*A\soliddotimesLambda\pr_2^*B)\cong f_\natural A\soliddotimesLambda g_\natural B$.
\end{lem*}
\begin{proof}
  By \cite[Proposition VII.3.1.(i)]{FS21} and \cite[Proposition VII.3.1.(iii)]{FS21}, we have natural isomorphisms
  \begin{align*}
    \pr_{2\natural}(\pr_1^*A\soliddotimesLambda\pr_2^*B)\cong (\pr_{2\natural}\pr_1^*A)\soliddotimesLambda B \cong (g^*f_\natural A)\soliddotimesLambda B.
  \end{align*}
  Applying $g_\natural$ to both sides and using \cite[Proposition VII.3.1.(i)]{FS21} again yield
  \begin{gather*}
h_\natural(\pr_1^*A\soliddotimesLambda\pr_2^*B)\cong g_\natural((g^*f_\natural A)\soliddotimesLambda B)\cong f_\natural A\soliddotimesLambda g_\natural B.\qedhere
  \end{gather*}
\end{proof}

\begin{rem}\label{rem:fetalgebras}
  The results of \S\ref{s:FFcurves}--\S\ref{s:plectic} generalize to finite \'etale $Q$-algebras $F$. We explain how this works when $I$ is a singleton. Write $F$ as $\textstyle\prod_{l=1}^cF_l$, where $F_l$ is a degree $d_l$ finite separable extension of $Q$.
  \begin{enumerate}[a)]
  \item We have a natural Cartesian square
    \begin{align*}
      \xymatrix{\coprod_{l=1}^cX_{S,F_l}\ar[d]\ar[r]^-m & X_{S,Q}\ar[d]\\
      \Spa F\ar[r] & \Spa Q,}
    \end{align*}
    which induces a natural closed embedding $m^{-1}:\Div^1_Q\ra\textstyle\prod_{l=1}^c\Div^{d_l}_{F_l}$.
  \item Let $H$ be a reductive group scheme over $F$ with geometrically connected fibers. Write $H$ as $\coprod_{l=1}^cH_l$, where $H_l$ is a connected reductive group over $F_l$, and let $S_l$ be a maximal subtorus of $H_l$ over $F_l$. Then $T=\prod_{l=1}^c\R_{F_l/Q}S_l$ is a maximal subtorus of $G=\R_{F/Q}H$, and we can identify $X_*^+(T)$ with the product of the $W_Q$-sets induced from the $W_{F_l}$-sets $X_*^+(S_l)$. Hence the $W_Q$-action on $X_*^+(T)$ naturally extends to a $W^{\plec}_{F/Q}\coloneqq\prod_{l=1}^cW^{\plec}_{F_l/Q}$-action via the canonical map $W_Q\ra W^{\plec}_{F/Q}$.

  \item Let $\mu$ be in $X_*^+(T)$, and write $W^{[\mu]}_{F/Q}$ for its stabilizer in $W^{\plec}_{F/Q}$. Then there exists an enumeration of representatives for $W_Q/W_{F_l}$ in $W_Q$ for all $1\leq l\leq c$ such that, under the resulting identification $W^{\plec}_{F/Q}\cong\prod_{l=1}^cW_{F_l}^{d_l}\rtimes\fS_{d_l}$, the subgroup $W^{[\mu]}_{F/Q}$ corresponds to $\prod_{l=1}^c\prod_{j=1}^{k_l}W_{F_{l,j}}^{\de_{l,j}}\rtimes\fS_{\de_{l,j}}$, where the $\de_{l,1},\dotsc,\de_{l,k_l}$ is a partition of $[d_l]$, and the $F_{l,1},\dotsc,F_{l,k_l}$ are finite separable extensions of $F_l$.

  \item Write $E$ for the field of definition of $\mu$, and write $\mu_l$ for the component of $\mu$ in the factor induced from $X_+^*(S_l)$. Write $b_l$ for the component of $b$ in the factor $B(H_l)$, let $K_l$ be a compact open subgroup of $H_l(F_l)$, and take $K=\prod_{l=1}^cK_l$. Then we have a natural Cartesian square
    \begin{align*}
      \xymatrix{\Sht_{G,b,\leq\mu,K}\ar[r]^-{m^{-1}}\ar[d]^-{f_K} & \prod_{l=1}^c\Sht^{(d_l)}_{H_l,b_l,\leq\mu_l,K_l}\ar[d]^-{\prod_{l=1}^cf_K} \\
      \Div^1_E\ar[r]^-{m^{-1}} & \prod_{l=1}^c\prod_{j=1}^{k_l}\Div^{\#\de_{l,j}}_{F_{l,j}}.
      }
    \end{align*}
  \item By using Theorem \ref{ss:maintheorem}, part d), and Lemma \ref{ss:Kunneth}, one proves that the object $f_{K\natural}(\prescript{\prime}{}{\IC}_\mu)$ of $D_{\mathrm{lis}}([*/\ul{G_b(Q)}],\La)^{BW_E}$ canonically lifts to an object of $D_{\mathrm{lis}}([*/\ul{G_b(Q)}],\La)^{B\prod_{l=1}^c\prod_{j=1}^{k_l}W^{\de_{l,j}}_{F_{l,j}}\rtimes\fS_{\de_{l,j}}}$. Consequently, when in the situation of \ref{ss:localShimuravarieties}, the complex of compactly supported cohomology of $\breve\cM_{G,b,\mu,K}$ with coefficients in $\La$ also canonically lifts to an object of $D_{\mathrm{lis}}([*/\ul{G_b(\bQ_p)}],\La)^{B\prod_{l=1}^c\prod_{j=1}^{k_l}W^{\de_{l,j}}_{F_{l,j}}\rtimes\fS_{\de_{l,j}}}$.
  \end{enumerate}
\end{rem}

\section{Applications to Shimura varieties}\label{s:Shimuravarieties}
At this point, we shift focus to a global context. Our goal in this section is to prove Theorem A and Theorem B. We start by describing the relation between decomposition groups and plectic decomposition groups. Next, we recall $p$-adic uniformization results of Rapoport--Zink and Varshavsky for certain very special Shimura varieties, as well as analogous results from Appendix \ref{s:appendix} for the basic loci of unramified abelian type Shimura varieties. We conclude by using these results to prove Theorem A and Theorem B.

\subsection{}\label{ss:localglobalplecticgroup}
Let $F$ be a degree $d$ finite extension of $\bQ$, and write $\ov\bQ$ for the algebraic closure of $\bQ$ in $\bC$. Write $F_p$ for $F\otimes_\bQ\bQ_p$, and recall that $F_p\cong\prod_u F_u$, where $u$ runs over places of $F$ above $p$. Write $\Ga^{\plec}_{F_p/\bQ_p}$ for $\Aut_{F_p}(F\otimes_\bQ\ov\bQ_p)$, and fix an isomorphism $\eta:\bC\ra^\sim\bC_p$.
\begin{prop*}
  We have a natural Cartesian square
  \begin{align*}
    \xymatrix{\Ga_{\bQ_p}\ar[r]\ar[d] & \Ga^{\plec}_{F_p/\bQ_p}\ar[d] \\
    \Ga_\bQ\ar[r] & \Ga^{\plec}_{F/\bQ},
    }
  \end{align*}
where each arrow is a continuous injective homomorphism. Moreover, $\Ga^{\plec}_{F_p/\bQ_p}$ is naturally identified with $\prod_u\Ga^{\plec}_{F_u/\bQ_p}$.
\end{prop*}
\begin{proof}
  Note that $F\otimes_\bQ\ov\bQ$ is the integral closure of $F$ in $F\otimes_\bQ\ov\bQ_p$. Hence restriction to $F\otimes_\bQ\ov\bQ$ yields the right arrow. The top arrow is given by $\ga\mapsto\id_F\otimes_\bQ\ga$, from which commutativity is immediate. Next, the natural identification $\Ga^{\plec}_{F_p/\bQ_p}\cong\prod_u\Ga^{\plec}_{F_u/\bQ_p}$ follows from the $F_p$-linear identification
  \begin{gather*}
    F\otimes_\bQ\ov\bQ_p\cong F_p\otimes_{\bQ_p}\ov\bQ_p\cong\textstyle\prod_uF_u\otimes_{\bQ_p}\ov\bQ_p.
  \end{gather*}
Finally, let $\de$ lie in $\Ga^{\plec}_{F_p/\bQ_p}$, and suppose that its restriction to $F\otimes_\bQ\ov\bQ$ lies in the image of $\Ga_\bQ$. Then $\de$ preserves $\ov\bQ$ in $F\otimes_\bQ\ov\bQ$, and its $\bQ_p$-linearity shows that $\de$ also preserves $\ov\bQ_p$ in $F\otimes_\bQ\ov\bQ_p$. Hence $\de$ arises from an element of $\Ga_{\bQ_p}$, as desired.
\end{proof}

\subsection{}\label{ss:localglobalplecticreflexgroup}
Let $H$ be a connected reductive group over $F$, and write $G\coloneqq\R_{F/\bQ}H$. Let $S$ be a maximal subtorus of $H$ over $F$, write $T\coloneqq\R_{F/\bQ}S$, and let $\mu$ be in $X_*^+(T)$. Write $E$ for the field of definition of $\mu$, and write $E_v$ for the field of definition of $\mu_{\bC_p}$.
\begin{cor*}
  We have a natural Cartesian square
  \begin{align*}
    \xymatrix{\Ga_{E_v}\ar[r]\ar[d] & \Ga^{[\mu]}_{F_p/\bQ_p}\ar[d]\\
    \Ga_E\ar[r] & \Ga^{[\mu]}_{F/\bQ},}
  \end{align*}
and $E_v$ is the closure of $E$ in $\bC_p$.
\end{cor*}
\begin{proof}
Because $\Ga_E$ (respectively $\Ga_{E_v}$) equals the stabilizer of $\mu$ in $\Ga_\bQ$ (respectively $\Ga_{\bQ_p}$), this follows immediately from Proposition \ref{ss:localglobalplecticgroup}.
\end{proof}

\subsection{}\label{ss:Shimuravarieties}
Maintain the notation of \ref{ss:localglobalplecticreflexgroup}, and assume that $F$ is totally real. We will consider two situations: the \emph{Drinfeld} case in \ref{ss:Drinfeldcase}, and the \emph{unramified abelian} case in \ref{ss:unramifiedabelian}. In both situations, we will have a Shimura datum $(G,X)$, and take $\mu$ to be its inverse Hodge cocharacter. For any compact open subgroup $K$ of $G(\bA_{f})$, write $\Sh_K(G,X)$ for the associated Shimura variety over $E$, and write $\Sh_K(G,X)^{\an}_{\breve{E_v}}$ for its rigid analytification over $\breve{E}_v$.

Write $b$ for the unique basic element of $B(G_{\bQ_p},\mu_{\bC_p})$, and recall that $G$ has a canonical inner form $G'$ over $\bQ$ such that
\begin{enumerate}[a)]
\item $G'_{\bQ_p}$ is isomorphic to $(G_{\bQ_p})_b$ over $\bQ_p$,
\item $G'_{\bA_f^p}$ is isomorphic to $G_{\bA_f^p}$ over $\bA_f^p$,
\item $G'_\bR$ is anisotropic modulo center over $\bR$
\end{enumerate}
\cite[Proposition 3.1]{Han21}. 

\subsection{}\label{ss:Drinfeldcase}
Let us describe the Drinfeld case, which refers to one of the following two subcases:
\begin{enumerate}[i)]
\item Let $\De$ be a quaternion algebra over $F$, and suppose that $\De$ splits at a nonempty set of archimedean places $\infty_1,\dotsc,\infty_t$ of $F$. For all $1\leq i\leq t$, the composition $\eta\circ\infty_i$ induces a place $u_i$ of $F$ above $p$. Suppose that the places $u_1,\dotsc,u_t$ are distinct and that $\De$ is not split at said places.

  Take $H$ to be the unit group $\De^\times$ over $F$, and take $(G,X)$ to be the associated Shimura datum as in \cite[5.1]{Var98}.

\item Let $F_c$ be a totally imaginary quadratic extension of $F$, let $\De$ be a central simple algebra over $F_c$ of dimension $n^2$, and let $*$ be an involution of the second kind on $\De$ over $F$. Suppose that $*$ has signature $(n-1,1)$ at a nonempty set of archimedean places $\infty_1,\dotsc,\infty_t$ of $F$ and signature $(n,0)$ at all other archimedean places of $F$, For all $1\leq i\leq d$, fix an embedding $\ov\infty_i:F_c\ra\bC$ that extends $\infty_i:F\ra\bR$, so that $\eta\circ\ov\infty_i$ induces a place $w_i$ of $F_c$ above $p$. Suppose that the places of $F$ below the places $w_i$ are distinct and split in $F_c$, and suppose that $\De$ has Hasse invariant $1/n$ at the places $w_i$.

Take $H$ to be the general unitary group $\GU(\De,*)$ over $F$, and take $(G,X)$ to be the associated Shimura datum as in \cite[2.2]{Var98}.
\end{enumerate}
In both subcases, note that $\Sh_K(G,X)$ is projective over $E$.

\subsection{}\label{ss:Varshavskyuniformization}
Maintain the notation of \ref{ss:Drinfeldcase}, and  recall the following $p$-adic uniformization result of Rapoport--Zink and Varshavsky. Let $K_p$ be a maximal compact open subgroup of $G(\bQ_p)$, and let $K^p$ be a compact open subgroup of $G(\bA_{f}^p)$.
\begin{thm*}
  We have an isomorphism
\begin{align*}
    G'(\bQ)\bs(\breve\cM_{G_{\bQ_p},b,\mu_{\bC_p},K_p}\times\ul{G(\bA_{f}^p)/K^p})\ra^\sim\Sh_{K_pK^p}(G,X)^{\an}_{\breve{E}_v}
\end{align*}
of rigid spaces over $\breve{E}_v$ that is compatible the Weil descent datum and varying $K^p$.
\end{thm*}
\begin{proof}
  Note that the factors of $(G_{\bQ_p},b,\mu_{\bC_p})$ where $\mu_{\bC_p}\neq0$ are of the form considered in \cite[(3.54)]{RZ96}, modulo twisting by a central character. For these factors, Theorem \ref{ss:RapoportZinkspaces}, \cite[theorem (3.72)]{RZ96}, and compatibility with twisting by central characters show that the local Shimura variety at level $K_p$ is precisely Drinfeld upper half space. As for the factors where $\mu_{\bC_p}=0$, we use \cite[Proposition 23.3.1]{SW20} and the proof of \cite[Proposition 23.2.1]{SW20} to identify the associated tower of local Shimura varieties with the appropriate discrete spaces. With these identifications, the desired result is \cite[Theorem 5.3]{Var98} in subcase i) and \cite[First Main Theorem 2.13]{Var98} in subcase ii).
\end{proof}

\subsection{}\label{ss:Varshavskydecompositionplectic}
Maintain the notation of \ref{ss:Varshavskyuniformization}. We now prove Theorem B.
\begin{thm*}
The complex $R\Ga(\Sh_{K_pK^p}(G,X)_{\ov\bQ_p},\La)$ canonically lifts from an object of $D^b(\Ga_{E_v},\La)$ to an object of $D^b(\Ga^{[\mu]}_{F_p/\bQ_p},\La)$.
\end{thm*}
\begin{proof}
Begin by replacing $\ov\bQ_p$ with $\bC_p$. Next, \cite[(3.8.1)]{Hub96} identifies
\begin{align*}
R\Ga(\Sh_{K_pK^p}(G,X)_{\bC_p},\La)\ra^\sim R\Ga(\Sh_{K_pK^p}(G,X)_{\bC_p}^{\an},\La).
\end{align*}
By taking derived invariants, we may assume $K_pK^p$ is sufficiently small. Then $G'(\bQ)$ acts properly discontinuously on $\breve\cM_{G_{\bQ_p},b,\mu_{\bQ_p},K_p}\times_{\breve{E}_v}\ul{G(\bA_f^p)/K^p}$ \cite[Lemma 2.9 a)]{Var98}, so Theorem \ref{ss:Varshavskyuniformization} and the K\"unneth formula yield an isomorphism from $R\Ga(\Sh_{K_pK^p}(G,X)^{\an}_{\bC_p},\La)$ to
\begin{align}\label{eq:Kunnethresult}
R\Ga(G'(\bQ),R\Ga_c(\breve\cM_{G_{\bQ_p},b,\mu_{\bQ_p},K_p},\La)\otimes^\bL_\La C_c(G(\bA_f^p)/K^p,\La)).\tag{$\star$}
\end{align}
Consider the image of (\ref{eq:Kunnethresult}) in $D^b(W_{E_v},\La)$. Remark \ref{rem:fetalgebras}.e) canonically lifts this to an object of $D^b(W^{[\mu]}_{F_p/\bQ_p},\La)$, and since we are calculating the cohomology of an algebraic variety, the underlying object of $D^b(\La)$ is a perfect complex. Because $\Ga^{[\mu]}_{F_p/\bQ_p}$ is the profinite completion of $W^{[\mu]}_{F_p/\bQ_p}$, our object of $D^b(W^{[\mu]}_{F_p/\bQ_p},\La)$ must lie in the full subcategory $D^b(\Ga^{[\mu]}_{F_p/\bQ_p},\La)$, as desired.
\end{proof}
\begin{rem}
Maintain the notation of \ref{ss:Drinfeldcase}, and suppose we are in subcase ii). When all places of $F$ above $p$ are unramified in $F_c$, Deligne's theory of connected components and results of Rapoport--Zink \cite[theorem (6.36)]{RZ96} yield a generalization of Theorem \ref{ss:Varshavskyuniformization} that allows arbitrary compact open subgroups $K_p$ of $G(\bQ_p)$. This gives a corresponding generalization of Theorem B, with exactly the same proof.
\end{rem}

\subsection{}\label{ss:unramifiedabelian}
Let us describe the unramified abelian case. Choose a reductive group scheme over $\bZ_{(p)}$ with geometrically connected fibers whose generic fiber is isomorphic to $G$. By abuse of notation, we also write $G$ for this group scheme. Suppose that $p\neq2$, and take $(G,X)$ to be a Shimura datum of abelian type.

\subsection{}\label{ss:unramifiedabelianuniformization}
Maintain the notation of \ref{ss:unramifiedabelian}. In this case, we have the following $p$-adic uniformization result. Let $K^p$ be a compact open subgroup of $G(\bA_{f}^p)$. Write $\Sh_{G(\bZ_p)K^p}(G,X)^b_{\breve{E}_v}$ for the basic locus as in \ref{ss:hyperspecialleveluniformization}, which is an open rigid subspace of $\Sh_{G(\bZ_p)K^p}(G,X)_{\breve{E}_v}^{\an}$ preserved by the Weil descent datum over $\breve{E}_v$.
\begin{thm*}
  We have an isomorphism
  \begin{align*}
    G'(\bQ)\bs(\breve\cM_{G_{\bQ_p},b,\mu_{\bC_p},G(\bZ_p)}\times\ul{G(\bA_f^p)/K^p})\ra^\sim\Sh_{G(\bZ_p)K^p}(G,X)^b_{\breve{E}_v}
  \end{align*}
of rigid spaces over $\breve{E}_v$ that is compatible with the Weil descent datum and varying $K^p$.
\end{thm*}
\begin{proof}
  Adapt the notation of \ref{ss:hyperspecialleveluniformization}. Theorem \ref{thm:uniformization} yields an isomorphism
  \begin{align*}
    \breve\fU(G,X)\ra^\sim\fS\fh_{G(\bZ_p)}(G,X)^b
  \end{align*}
  of formal schemes over $\breve\cO_{E_v}$ that is compatible with the Weil descent datum and the $G(\bA_f^p)$-action. By \ref{ss:integrallocalshimuravarieties}, taking adic generic fibers gives an isomorphism
  \begin{align*}
    \varprojlim_{K^p}\left[G'(\bQ)\bs(\breve\cM_{G_{\bQ_p},b,\mu_{\bC_p},G(\bZ_p)}\times\ul{G(\bA_f^p)/K^p})\right]\ra^\sim\varprojlim_{K^p}\left[\Sh_{G(\bZ_p)K^p}(G,X)^b_{\breve{E}_v}\right]
  \end{align*}
of rigid spaces over $\breve{E}_v$. Thus quotienting by $K^p$ yields the desired result.
\end{proof}
\subsection{}\label{ss:unramifiedabeliandecompositionplectic}
Maintain the notation of \ref{ss:unramifiedabelianuniformization}. We now prove Theorem A.
\begin{thm*}
The complex $R\Ga_c(\Sh_{G(\bZ_p)K^p}(G,X)_{\bC_p}^b,\La)$ canonically lifts from an object of $D^b(\Ga_{E_v},\La)$ to an object of $D^b(\Ga^{[\mu]}_{F_p/\bQ_p},\La)$.
\end{thm*}
\begin{proof}
Using \cite[(3.5.17)]{Hub96} to show that the object $R\Ga_c(\Sh_{G(\bZ_p)K^p}(G,X)_{\bC_p}^b,\La)$ of $D^b(\La)$ is a perfect complex, this follows from arguing exactly as in the proof of Theorem \ref{ss:Varshavskydecompositionplectic}.
\end{proof}

\appendix
\section{Basic uniformization for unramified abelian type Shimura varieties at hyperspecial level}\label{s:appendix}
Our goal for the appendix is to prove Theorem D. First, we recall the theory of integral local Shimura varieties. Using their relation with affine Deligne--Lusztig varieties, we describe their sets of connected components via results of Chen--Kisin--Viehmann and Gleason. Next, we interweave this description with work of Kisin and Pappas--Rapoport to develop a theory of connected components for basic $p$-adic uniformization at hyperspecial level. We conclude by applying this theory to deduce Theorem D from Kim's results in the unramified Hodge type case.

\subsection{}\label{ss:integrallocalshimuravarieties}
Let $G$ a reductive group scheme over $\bZ_p$ with geometrically connected fibers. By abuse of notation, we also write $G$ for its generic fiber. Let $(G,b,\mu)$ be a local Shimura datum over $\bQ_p$, and write $E$ for the field of definition of $\mu$, which is unramified over $\bQ_p$. Write $r$ for $[E:\bQ_p]$, and write $q$ for $p^r$.

Write $\breve\fM_{G,b,\mu}$ for the small  v-sheaf over $\Spd\breve\cO_E$ as in \cite[Definition 25.1.1]{SW20}. Recall that $\breve\fM_{G,b,\mu}$ enjoys a Weil descent datum $\Phi$ over $\breve\cO_E$ as well as an action of $G_b(\bQ_p)$, and its fiber over $\Spd\breve{E}$ is naturally isomorphic to $\breve\cM_{G,b,\mu,G(\bZ_p)}$. Moreover, \cite[Proposition 2.30]{Gle21} identifies its reduction with the affine Deligne--Lusztig variety $X_\mu^G(b)$ as in \cite[Section 3.1.1]{Zhu17b}. 

When $p\neq2$ and $(G,b,\mu)$ is of abelian type, $\breve\fM_{G,b,\mu}$ is the diamond associated with a unique normal formal scheme that is flat and locally formally of finite type over $\breve\cO_E$ by \cite[Proposition 18.4.1]{SW20} and \cite[Theorem 2.5.4]{PR22}. By abuse of notation, we also write $\breve\fM_{G,b,\mu}$ for this formal scheme.

\subsection{}\label{ss:integrallocalshimuravarietiesconnectedcomponents}
Write $w_G$ for the homomorphism $G(\breve\bQ_p)\ra\pi_1(G)$ as in \cite[(7.1.1)]{Kot97}. Its restriction to $G_b(\bQ_p)$ surjects onto $\pi_1(G)^\vp$ \cite[(4.6.4)]{Kis17}\footnote{While this is stated for $(G,\mu)$ arising from a Shimura datum of Hodge type, the proof applies verbatim in our situation.}, and write $G_b(\bQ_p)^+$ for the kernel of this restriction.

We can describe the sets of connected components of $\breve\fM_{G,b,\mu}$ as follows.
\begin{prop*}
  Assume that $b$ is basic. Then
  \begin{enumerate}[a)]
  \item $G_b(\bQ_p)$ acts transitively on $\pi_0(\breve\fM_{G,b,\mu})$ with stabilizer $G_b(\bQ_p)^+$, so
    \begin{align*}
      \pi_0(\breve\fM_{G,b,\mu})\cong \pi_1(G)^\vp,
    \end{align*}
  \item $\Phi$ acts on $\pi_0(\breve\fM_{G,b,\mu})$ via translation by the image of
    \begin{align*}
      \mu+\vp(\mu)+\dotsb+\vp^{r-1}(\mu).
    \end{align*}
  \end{enumerate}
In particular, each connected component $\breve\fM_{G,b,\mu}^+$ of $\breve\fM_{G,b,\mu}$ is isomorphic.
\end{prop*}
\begin{proof}
Because $b$ is basic, $(G,b,\mu)$ is Hodge--Newton irreducible. Therefore everything follows from \cite[Theorem A.4]{HZ20} and \cite[Corollary 4.1.16]{CKV15}.
\end{proof}

\subsection{}\label{ss:pi1invariants}
We will need the following description of $\pi_1(G)^\vp$, which has already been used in the proof of \cite[(3.8.5)]{Kis17}. Recall that the image of $G^{\simplyc}(\bQ_p)$ in $G(\bQ_p)$ contains the commutator subgroup of $G(\bQ_p)$ \cite[2.0.2]{Del79}.
\begin{lem*}
We have a natural isomorphism $G^{\simplyc}(\bQ_p)\bs G(\bQ_p)/G(\bZ_p)\ra^\sim\pi_1(G)^\vp$.
\end{lem*}
\begin{proof}
  First, assume that $G^{\der}$ is simply connected. Then Kneser's theorem gives an isomorphism $G^{\simplyc}(\bQ_p)\bs G(\bQ_p)\ra^\sim G^{\ab}(\bQ_p)$, and Lang's lemma implies that $G(\bZ_p)$ surjects onto $G^{\ab}(\bZ_p)$. Since $G^{\der}$ is simply connected, the evaluation map induces an isomorphism $(\pi_1(G)\otimes_\bZ\breve\bQ_p^\times)^\vp\ra^\sim G^{\ab}(\bQ_p)$ that identifies $(\pi_1(G)\otimes_\bZ\breve\bZ_p^\times)^\vp$ with $G^{\ab}(\bZ_p)$. Finally, the valuation map yields
  \begin{align*}
    (\pi_1(G)\otimes_\bZ\breve\bQ_p^\times)^\vp/(\pi_1(G)\otimes_\bZ\breve\bZ_p^\times)^\vp\ra^\sim\pi_1(G)^\vp.
  \end{align*}
  
  Next, consider general $G$. Since $G$ splits over an unramified extension of $\bQ_p$, there exists a $z$-extension $G'\ra G$ whose kernel $Z'$ is a product of unramified induced tori \cite[V.3.1]{DMOS82}. Lang's lemma implies that $G'(\bZ_p)$ surjects onto $G(\bZ_p)$, so Hilbert 90 and Shapiro's lemma give a surjection $G'^{\ab}(\bQ_p)/G'^{\ab}(\bZ_p)\ra G^{\simplyc}(\bQ_p)\bs G(\bQ_p)/G(\bZ_p)$ whose kernel equals the image of $Z'(\bQ_p)$. The sequence $\pi_1(Z')^\vp\ra\pi_1(G')^\vp\ra\pi_1(G)^\vp$ is short exact by Shapiro's lemma, and under our identification $G'^{\ab}(\bQ_p)/G'^{\ab}(\bZ_p)\ra^\sim\pi_1(G')^\vp$, the image of $Z'(\bQ_p)$ is precisely $\pi_1(Z')^\vp$. Hence the result follows.
\end{proof}

\subsection{}
We now switch our notation to a global context. Let $G$ be a reductive group scheme over $\bZ_{(p)}$ with geometrically connected fibers. By abuse of notation, we also write $G$ for its generic fiber. Let $(G,X)$ be a Shimura datum, write $\mu$ for its inverse Hodge cocharacter, and write $E$ for the field of definition of $\mu$. For any compact open subgroup $K$ of $G(\bA_f)$, write $\Sh_K(G,X)$ for the associated Shimura variety over $E$.

Fix an isomorphism $\bC\ra^\sim\bC_p$, write $v$ for the resulting place of $E$ over $p$, and note that the closure $E_v$ of $E$ in $\bC_p$ equals the field of definition of $\mu_{\bC_p}$. Write $b$ for the basic element of $B(G_{\bQ_p},\mu_{\bC_p})$. Recall that $G$ has a canonical inner form $G'$ over $\bQ$ such that
\begin{enumerate}[a)]
\item $G'_{\bQ_p}$ is isomorphic to $(G_{\bQ_p})_b$ over $\bQ_p$,
\item $G'_{\bA_f^p}$ is isomorphic to $G_{\bA_f^p}$ over $\bA_f^p$,
\item $G'_\bR$ is anisotropic modulo center over $\bR$
\end{enumerate}
\cite[Proposition 3.1]{Han21}. This enables us to form the small v-sheaf
\begin{align*}
\breve\fU(G,X)\coloneqq\varprojlim_{K^p}\left[G'(\bQ)\bs(\breve\fM_{G_{\bZ_p},b,\mu_{\bC_p}}\times\ul{G(\bA_f^p)/K^p})\right],
\end{align*}
where $K^p$ runs through compact open subgroups of $G(\bA_f^p)$. This is meant to be the source of the uniformization morphism. Note that $\breve\fU(G,X)$ inherits a Weil descent datum $\Phi$ over $\breve\cO_{E_v}$ as well as an action of $\ul{G(\bA_f^p)}$.

\subsection{}\label{ss:connectedcomponentsmap}
The following is meant to correspond to the connected components morphism of the relevant Shimura variety. Write $\sA_p(G)$ for the topological group as in \cite[(3.3.2)]{Kis10}, and write $\sA_p(G)^\circ$ for its closed subgroup as in \cite[(3.3.2)]{Kis10}.\footnote{These groups are denoted $\sA(G_{\bZ_{(p)}})$ and $\sA(G_{\bZ_{(p)}})^\circ$ in \cite{Kis10}, \cite{Kis17}, and \cite{SZ17}.} Recall that $\sA_p(G)^\circ$ only depends on $G^{\der}$, and $\sA_p(G)^\circ$ contains the image of $G^{\simplyc}(\bA_f^p)$ \cite[2.0.9]{Del79}.

The proof of \cite[(3.7.2)]{Kis17} yields a natural action of $\ul{\sA_p(G)}$ on $\breve\fU(G,X)$.
\begin{prop*}
  We have a natural $\ul{\sA_p(G)}$-equivariant morphism over $\Spd\breve\cO_{E_v}$
  \begin{align*}
    \fc_G:\breve\fU(G,X)\ra\ul{\sA_p(G)^\circ\bs\sA_p(G)}.
  \end{align*}
Moreover, $\fc_G\circ\Phi$ equals $\fc_G$ postcomposed with translation by the image of $\mu+\vp(\mu)+\dotsb+\vp^{r-1}(\mu)$ in $\sA_p(G)$.
\end{prop*}
\begin{proof}
  Recall that the natural map $\sA_p(G)^\circ\bs\sA_p(G)\ra G(\bQ)^-_+\bs G(\bA_f)/G(\bZ_p)$ is an isomorphism \cite[(3.3.3)]{Kis10}, where $G(\bQ)^-_+$ denotes the closure in $G(\bA_f)$ of the preimage of $G^{\ad}(\bR)^\circ$ in $G(\bQ)$.  Proposition \ref{ss:integrallocalshimuravarietiesconnectedcomponents} and Lemma \ref{ss:pi1invariants} indicate that the connected components morphism of $\breve\fM_{G_{\bZ_p},b,\mu_{\bC_p}}$ is of the form
  \begin{align*}
\pi_0:\breve\fM_{G_{\bZ_p},b,\mu_{\bC_p}}\ra\ul{G^{\simplyc}(\bQ_p)\bs G(\bQ_p)/G(\bZ_p)}.
  \end{align*}
Using this, we can consider the morphism
\begin{align*}
\fc_{G,K^p}:\breve\fM_{G_{\bZ_p},b,\mu_{\bC_p}}\times\ul{G(\bA_f^p)/K^p}\ra\ul{G(\bQ)_+^-\bs G(\bA_f)/G(\bZ_p)K^p}
\end{align*}
that sends $(x,g^pK^p)$ to the image of $(\pi_0(x),g^pK^p)$, where we use the fact that $G(\bQ)_+^-$ contains the image of $G^{\simplyc}(\bQ_p)$.

By construction, the restriction of $\fc_G$ to the images of $\breve\fM_{G_{\bZ_p},b,\mu_{\bC_p}}$ depend only on its connected components. The connected components of $\breve\fM_{G_{\bZ_p},b,\mu_{\bC_p}}$ equal those of its reduction, so \cite[(3.7.4)]{Kis17} and \cite[(3.6.10)]{Kis17} imply that $\fc_{G,K^p}$ factors through $G'(\bQ)\bs(\breve\fM_{G_{\bZ_p},b,\mu_{\bC_p}}\times\ul{G'(\bA_f^p)/K^p})$.

Finally, form the inverse limit
\begin{align*}
\varprojlim_{K^p}\fc_{G,K^p}:\breve\fU(G,X)\ra\ul{G(\bQ)^-_+\bs G(\bA_f)/G(\bZ_p)},
\end{align*}
and let $\fc_G$ by its postcomposition with the inverse of
\begin{align*}
\ul{\sA_p(G)^\circ\bs\sA_p(G)}\ra^\sim \ul{G(\bQ)^-_+\bs G(\bA_f)/G(\bZ_p)}.
\end{align*}
Proposition \ref{ss:integrallocalshimuravarietiesconnectedcomponents}.b) shows $\fc_G$ satisfies the desired compatibility with $\Phi$.
\end{proof}

\subsection{}\label{ss:actualsurjectionindependence}
The connected components of the uniformization morphism should depend only on $G^{\der}$ and $X^+$. To make this precise, we need the following lemma. Let $G_1$ be another reductive group scheme over $\bZ_{(p)}$ with geometrically connected fibers, let $(G_1,X_1)$ be another Shimura datum, and let $G_1\ra G$ be a surjective morphism over $\bZ_{(p)}$ that induces an isomorphism $G^{\der}_1\ra^\sim G^{\der}$ and a morphism $(G_1,X_1)\ra(G,X)$ of Shimura data.

By Proposition \ref{ss:connectedcomponentsmap}, each fiber $\breve\fU(G,X)^+$ of $\fc_G$ is isomorphic.
\begin{lem*}
The natural morphism $\breve\fU(G_1,X_1)^+\ra\breve\fU(G,X)^+$ is an isomorphism.
\end{lem*}
\begin{proof}
  Note that the connected components $\breve\fM_{G_{\bZ_p},b,\mu_{\bC_p}}^+$ map isomorphically onto their images in $\breve\fU(G,X)$. As $\varprojlim_{K^p}\ul{G'(\bQ)\bs G(\bA_f^p)/K^p}$ is totally disconnected, these images are precisely the connected components of $\breve\fU(G,X)$ and hence $\breve\fU(G,X)^+$.

  The morphism $(G_{1\bZ_p},b_1,\mu_{1\bC_p})\ra(G_{\bZ_p},b,\mu_{\bC_p})$ is an ad-isomorphism of integral local shtuka data as in \cite[p.~31]{PR22}. Therefore \cite[Theorem 5.1.2]{PR22} implies that our natural morphism induces an isomorphism on connected components. Because the connected components of $\breve\fU(G,X)^+$ equal those of its reduction, we see from \cite[(3.8.2)]{Kis17} that our morphism also induces a bijection on $\pi_0$, which enables us to conclude.
\end{proof}

\subsection{}\label{ss:generalindependence}
For any multiple $s$ of $r$, write $\sE_p^s(G^{\der})$ for the stabilizer of $\breve\fU(G,X)^+$ in $\sA_p(G)\times\Phi^{\frac{s}r\bZ}$. When $s=r$, we omit it from our notation.
\begin{lem*}
The closed subgroup $\sE_p^s(G^{\der})$ is an extension of $\Phi^{\frac{s}r\bZ}$ by $\sA_p(G)^\circ$. Moreover, $\sE_p^s(G^{\der})$ and $\breve\fU(G,X)^+$ only depend on $G^{\der}$, $X^+$, and $s$.
\end{lem*}
\begin{proof}
Using Proposition \ref{ss:connectedcomponentsmap} and Lemma \ref{ss:actualsurjectionindependence}, this follows from arguing exactly as in the proof of \cite[(3.8.5)]{Kis17} and \cite[(3.8.6)]{Kis17}.
\end{proof}

\subsection{}\label{ss:inducedindependence}
We start by proving the main theorem for connected components of the uniformization morphism in the following special case. Let $G_2$ be another reductive group scheme over $\bZ_{(p)}$ with geometrically connected fibers, let $(G_2,X_2)$ be another Shimura datum, and let $G_2\ra G$ be a surjective morphism over $\bZ_{(p)}$ that induces a morphism $(G_2,X_2)\ra(G,X)$ of Shimura data and whose kernel is a product of induced tori. Assume that $s$ is a multiple of $r_2$.
\begin{lem*}
  We have a natural isomorphism
  \begin{align*}
  \sA_p(G)^\circ*_{\sA_p(G_2)^\circ}\sE_p^s(G^{\der}_2)\ra^\sim\sE_p^s(G^{\der}).
  \end{align*}
  Under this identification, we have a natural $\ul{\sE_p^s(G^{\der})}$-equivariant isomorphism $\ul{\sA_p(G)^\circ}\times^{\ul{\sA_p(G_2)^\circ}}\breve\fU(G_2,X)^+\ra^\sim\breve\fU(G,X)^+$.
\end{lem*}
\begin{proof}
  Proposition \ref{ss:connectedcomponentsmap} implies that we have a natural continuous homomorphism $\sA_p(G)^\circ*_{\sA_p(G_2)^\circ}\sE_p^s(G^{\der}_2)\ra\sE_p^s(G^{\der})$, and Lemma \ref{ss:generalindependence} indicates that it is an isomorphism. Proposition \ref{ss:connectedcomponentsmap} also implies that we have a natural $\ul{\sE_p^s(G^{\der})}$-equivariant morphism $\ul{\sA_p(G)^\circ}\times^{\ul{\sA_p(G_2)^\circ}}\breve\fU(G_2,X)^+\ra\breve\fU(G,X)^+$. Observe that the connected components $\breve\fM^+_{G_{2\bZ_p},b_2,\mu_{2\bC_p}}$ map isomorphically onto their images in $\ul{\sA_p(G)^\circ}\times^{\ul{\sA_p(G_2)^\circ}}\breve\fU(G_2,X)^+$, and these images are precisely its connected components. Using \cite[Theorem 5.1.2]{PR22} and \cite[(3.8.10)]{Kis17}, the desired result follows from arguing as in the proof of Lemma \ref{ss:actualsurjectionindependence}.
\end{proof}

\subsection{}\label{ss:mainequality}
We now bootstrap from Lemma \ref{ss:inducedindependence} to the general case. Let $G_3$ be another reductive group scheme over $\bZ_{(p)}$ with geometrically connected fibers, let $(G_3,X_3)$ be another Shimura datum, and let $G^{\der}_3\ra G^{\der}$ be an isogeny over $\bZ_{(p)}$ that induces an isomorphism $(G_3^{\ad},X^{\ad}_3)\ra^\sim(G^{\ad},X^{\ad})$. Assume that $s$ is a multiple of $r$ and $r_3$.
\begin{thm*}
  We have an isomorphism
  \begin{align*}
    \sA_p(G)^\circ*_{\sA_p(G_3)^\circ}\sE^s_p(G^{\der}_3)\cong\sE_p^s(G^{\der}).
  \end{align*}
  Under this identification, we have a $\ul{\sE_p^s(G^{\der})}$-equivariant isomorphism
  \begin{align*}
\ul{\sA_p(G)^\circ}\times^{\ul{\sA_p(G_3)^\circ}}\breve\fU(G_3,X_3)^+\cong\fU(G,X)^+
  \end{align*}
  and hence an $\ul{\sA_p(G)}\times\Phi^{\frac{s}r\bZ}$-equivariant isomorphism
  \begin{align*}
\ul{\sA_p(G)}\times^{\ul{\sA_p(G_3)^\circ}}\breve\fU(G_3,X_3)^+\cong\breve\fU(G,X).
  \end{align*}
\end{thm*}
\begin{proof}
By Lemma \ref{ss:generalindependence}, we may replace $G_3$ with the relative connected component of $G_3\times_{G^{\ad}}G$ over $\bZ_{(p)}$ and therefore assume that $G^{\der}_3\ra G^{\der}$ is induced by a surjective morphism $G_3\ra G$ over $\bZ_{(p)}$. Next, write $Z$ for the kernel of $G_3\ra G$, and embed $Z$ into a product of induced tori $Z'$ over $\bZ_{(p)}$. Lemma \ref{ss:generalindependence} enables us to replace $G_3$ with $G_3\times^ZZ'$ and thus additionally assume that $Z$ is a product of induced tori. Then the first two statements follow from Lemma \ref{ss:inducedindependence}, and the final statement follows from Proposition \ref{ss:connectedcomponentsmap}.
\end{proof}

\subsection{}\label{ss:hyperspecialleveluniformization}
Let us establish some notation on the basic locus. Assume that $p\neq2$ and $(G,X)$ is of abelian type. Write $\Sh_{G(\bZ_p)}(G,X)$ for the inverse limit
\begin{align*}
\varprojlim_{K^p}\Sh_{G(\bZ_p)K^p}(G,X),
\end{align*}
and write $\sS_{G(\bZ_p)}(G,X)$ for its integral canonical model over $\cO_{E(v)}$ as in \cite[(3.4.14)]{Kis10}. Write $\sS_{G(\bZ_p)}(G,X)^b_{\ov\bF_q}$ for the basic locus as in \cite[p.~183]{SZ17}, which is a closed subscheme of $\sS_{G(\bZ_p)}(G,X)_{\ov\bF_q}$ preserved by Frobenius as well as the $\sA_p(G)$-action. Write $\fS\fh_{G(\bZ_p)}(G,X)^b$ for the completion of $\sS_{G(\bZ_p)}(G,X)_{\breve\cO_{E_v}}$ along $\sS_{G(\bZ_p)}(G,X)_{\ov\bF_q}^b$, and write $\Sh_{G(\bZ_p)}(G,X)^b_{\breve{E}_v}$ for the adic generic fiber of $\fS_{G(\bZ_p)}(G,X)^b$ over $\breve{E}_v$. Note that $\fS\fh_{G(\bZ_p)}(G,X)^b$ inherits a Weil descent datum over $\breve\cO_{E_v}$ as well as an action of $\sA_p(G)$.

Write $\Sh_{G(\bZ_p)K^p}(G,X)^b_{\breve{E}_v}$ for the smooth rigid space $\Sh_{G(\bZ_p)}(G,X)^b_{\breve{E}_v}/K^p$.

\subsection{}\label{thm:uniformization}
We now prove Theorem D.
\begin{thm*}
  We have an isomorphism
  \begin{align*}
    \breve\fU(G,X)\ra^\sim\fS\fh_{G(\bZ_p)}(G,X)^b
  \end{align*}
that is compatible with the Weil descent datum and the $\sA_p(G)$-action.
\end{thm*}
\begin{proof}
  When $(G,X)$ is of Hodge type, this follows from Theorem \ref{ss:Kimspaces} and \cite[Theorem 4.7]{Kim18b}. Next, consider $(G,X)$ of abelian type. Then there exists a reductive group scheme $G_3$ over $\bZ_{(p)}$ with geometrically connected fibers, a Shimura datum $(G_3,X_3)$ of Hodge type, and an isogeny $G_3\ra G$ over $\bZ_{(p)}$ that induces an isomorphism $(G^{\ad}_3,X^{\ad}_3)\ra^\sim(G^{\ad},X^{\ad})$ such that $r_3$ divides $r$ \cite[(4.6.6)]{Kis17}. Combining \cite[2.1.16]{Del79}, \cite[(2.2.4)]{Kis10}, and \cite[(3.3.3)]{Kis10} shows that the connected components morphism of $\sS_{G(\bZ_p)}(G_3,X_3)_{\breve\cO_{E_{3v}}}$ is of the form
  \begin{align*}
   \pi_0:\sS_{G(\bZ_p)}(G_3,X_3)_{\breve\cO_{E_{3v}}}\ra\ul{\sA_p(G_3)^\circ\bs\sA_p(G_3)}.
  \end{align*}
Moreover, \cite[(3.6.10)]{Kis17} yields a commutative square
\begin{align*}
  \xymatrix{\fS\fh_{G_3(\bZ_p)}(G_3,X_3)^b\ar[r] & \sS_{G_3(\bZ_p)}(G_3,X_3)_{\breve\cO_{E_{3v}}}\ar[d]^-{\pi_0}\\
  \breve\fU(G_3,X_3)\ar[r]^-{\fc_{G_3}}\ar[u]^[left]{\sim} & \ul{\sA_p(G_3)^\circ\bs\sA_p(G_3)}.}
\end{align*}
Taking fibers gives us an $\sE_p^r(G^{\der}_3)$-equivariant isomorphism
\begin{align*}
\breve\fU(G_3,X_3)^+\ra^\sim\fS\fh_{G_3(\bZ_p)}(G_3,X_3)^{b,+},
\end{align*}
where $\fS\fh_{G_3(\bZ_p)}(G_3,X_3)^{b,+}$ denotes the intersection of $\fS\fh_{G_3(\bZ_p)}(G_3,X_3)^b$ with a connected component of $\sS_{G_3(\bZ_p)}(G_3,X_3)_{\breve\cO_{E_{3v}}}$. Theorem \ref{ss:mainequality} shows that we have an $\sA_p(G)\times\Phi^\bZ$-equivariant isomorphism
\begin{align*}
\breve\fU(G,X)\ra^\sim\ul{\sA_p(G)}\times^{\ul{\sA_p(G_3)^\circ}}\breve\fU(G_3,X_3)^+\ra^\sim\ul{\sA_p(G)}\times^{\ul{\sA_p(G_3)^\circ}}\fS\fh_{G_3(\bZ_p)}(G_3,X_3)^{b,+},
\end{align*}
and \cite[p.~183]{SZ17} implies that we have an $\sA_p(G)\times\Phi^{\bZ}$-equivariant isomorphism
\begin{gather*}
\ul{\sA_p(G)}\times^{\ul{\sA_p(G_3)^\circ}}\fS_{G_3(\bZ_p)}(G_3,X_3)^{b,+}\ra^\sim\fS\fh_{G(\bZ_p)}(G,X)^b.\qedhere
\end{gather*}
\end{proof}
\bibliographystyle{habbrv}
\bibliography{biblio}

\def\cprime{$'$}
\begin{thebibliography}{10}

\bibitem{Ai18}
X.~{Ai}.
\newblock {Generalized multiple zeta values over number fields I}.
\newblock {\em arXiv e-prints}, page arXiv:1809.07370, Sept. 2018, 1809.07370.

\bibitem{An22}
J.~Ansch\"{u}tz.
\newblock Extending torsors on the punctured {${\rm Spec}(A_{\inf})$}.
\newblock {\em J. Reine Angew. Math.}, 783:227--268, 2022.

\bibitem{BL84}
J.-L. Brylinski and J.-P. Labesse.
\newblock Cohomologie d'intersection et fonctions {$L$} de certaines
  vari\'{e}t\'{e}s de {S}himura.
\newblock {\em Ann. Sci. \'{E}cole Norm. Sup. (4)}, 17(3):361--412, 1984.

\bibitem{CKV15}
M.~Chen, M.~Kisin, and E.~Viehmann.
\newblock Connected components of affine {D}eligne-{L}usztig varieties in mixed
  characteristic.
\newblock {\em Compos. Math.}, 151(9):1697--1762, 2015.

\bibitem{DS20}
C.~Davidescu and A.~J. Scholl.
\newblock Extensions in the cohomology of {H}ilbert modular varieties.
\newblock {\em M\"unster J. Math.}, 13(2):509--518, 2020.

\bibitem{Del79}
P.~Deligne.
\newblock Vari\'{e}t\'{e}s de {S}himura: interpr\'{e}tation modulaire, et
  techniques de construction de mod\`eles canoniques.
\newblock In {\em Automorphic forms, representations and {$L$}-functions
  ({P}roc. {S}ympos. {P}ure {M}ath., {O}regon {S}tate {U}niv., {C}orvallis,
  {O}re., 1977), {P}art 2}, Proc. Sympos. Pure Math., XXXIII, pages 247--289.
  Amer. Math. Soc., Providence, R.I., 1979.

\bibitem{DMOS82}
P.~Deligne, J.~S. Milne, A.~Ogus, and K.-y. Shih.
\newblock {\em Hodge cycles, motives, and {S}himura varieties}, volume 900 of
  {\em Lecture Notes in Mathematics}.
\newblock Springer-Verlag, Berlin-New York, 1982.

\bibitem{FS21}
L.~{Fargues} and P.~{Scholze}.
\newblock {Geometrization of the local Langlands correspondence}.
\newblock {\em arXiv e-prints}, page arXiv:2102.13459, Jan. 2024, 2102.13459.

\bibitem{For24}
M.~Fornea.
\newblock Plectic {J}acobians.
\newblock {\em Q. J. Math.}, 75(3):789--808, 2024.

\bibitem{FG23}
M.~Fornea and L.~Gehrmann.
\newblock Plectic {S}tark-{H}eegner points.
\newblock {\em Adv. Math.}, 414:Paper No. 108861, 42, 2023.

\bibitem{FGM22}
M.~Fornea, X.~Guitart, and M.~Masdeu.
\newblock Plectic {$p$}-adic invariants.
\newblock {\em Adv. Math.}, 406:Paper No. 108484, 26, 2022.

\bibitem{Gle21}
I.~{Gleason}.
\newblock {On the geometric connected components of moduli spaces of p-adic
  shtukas and local Shimura varieties}.
\newblock {\em arXiv e-prints}, page arXiv:2107.03579, July 2021, 2107.03579.

\bibitem{Han21}
D.~Hansen.
\newblock {On the supercuspidal cohomology of basic local Shimura varieties}.
\newblock \url{http://davidrenshawhansen.net/middlev2.pdf}, 2024.

\bibitem{HZ20}
X.~He and R.~Zhou.
\newblock On the connected components of affine {D}eligne-{L}usztig varieties.
\newblock {\em Duke Math. J.}, 169(14):2697--2765, 2020.

\bibitem{Hub96}
R.~Huber.
\newblock {\em \'{E}tale cohomology of rigid analytic varieties and adic
  spaces}.
\newblock Aspects of Mathematics, E30. Friedr. Vieweg \& Sohn, Braunschweig,
  1996.

\bibitem{Kim18}
W.~Kim.
\newblock Rapoport-{Z}ink spaces of {H}odge type.
\newblock {\em Forum Math. Sigma}, 6:Paper No. e8, 110, 2018.

\bibitem{Kim18b}
W.~Kim.
\newblock Rapoport-{Z}ink uniformization of {H}odge-type {S}himura varieties.
\newblock {\em Forum Math. Sigma}, 6:Paper No. e16, 36, 2018.

\bibitem{Kis10}
M.~Kisin.
\newblock Integral models for {S}himura varieties of abelian type.
\newblock {\em J. Amer. Math. Soc.}, 23(4):967--1012, 2010.

\bibitem{Kis17}
M.~Kisin.
\newblock {${\rm mod}\,p$} points on {S}himura varieties of abelian type.
\newblock {\em J. Amer. Math. Soc.}, 30(3):819--914, 2017.

\bibitem{Kot85}
R.~E. Kottwitz.
\newblock Isocrystals with additional structure.
\newblock {\em Compositio Math.}, 56(2):201--220, 1985.

\bibitem{Kot90}
R.~E. Kottwitz.
\newblock Shimura varieties and {$\lambda$}-adic representations.
\newblock In {\em Automorphic forms, {S}himura varieties, and {$L$}-functions,
  {V}ol.\ {I} ({A}nn {A}rbor, {MI}, 1988)}, volume~10 of {\em Perspect. Math.},
  pages 161--209. Academic Press, Boston, MA, 1990.

\bibitem{Kot97}
R.~E. Kottwitz.
\newblock Isocrystals with additional structure. {II}.
\newblock {\em Compositio Math.}, 109(3):255--339, 1997.

\bibitem{Leo19}
M.~Leonhardt.
\newblock {\em {Plectic arithmetic of Hilbert modular varieties}}.
\newblock PhD thesis, University of Cambridge, 2019.
\newblock \url{https://doi.org/10.17863/CAM.49057}.

\bibitem{LH21}
S.~D. {Li-Huerta}.
\newblock The plectic conjecture over function fields, 2021.
\newblock \url{https://arxiv.org/abs/2106.05382}.

\bibitem{mythesis}
S.~D. Li-Huerta.
\newblock {\em On the {P}lectic {C}onjecture}.
\newblock ProQuest LLC, Ann Arbor, MI, 2023.
\newblock Thesis (Ph.D.)--Harvard University.

\bibitem{LZ25}
D.~Loeffler and S.~L. Zerbes.
\newblock Plectic structures in {$p$}-adic de {R}ham cohomology.
\newblock {\em J. Number Theory}, 270:238--259, 2025.

\bibitem{MR16}
B.~Mazur and K.~Rubin.
\newblock Controlling {S}elmer groups in the higher core rank case.
\newblock {\em J. Th\'{e}or. Nombres Bordeaux}, 28(1):145--183, 2016.

\bibitem{MV07}
I.~Mirkovi\'c and K.~Vilonen.
\newblock Geometric {L}anglands duality and representations of algebraic groups
  over commutative rings.
\newblock {\em Ann. of Math. (2)}, 166(1):95--143, 2007.

\bibitem{Nek18}
J.~Nekov\'{a}\v{r}.
\newblock Eichler-{S}himura relations and semisimplicity of \'{e}tale
  cohomology of quaternionic {S}himura varieties.
\newblock {\em Ann. Sci. \'{E}c. Norm. Sup\'{e}r. (4)}, 51(5):1179--1252, 2018.

\bibitem{NS16}
J.~Nekov\'{a}\v{r} and A.~J. Scholl.
\newblock Introduction to plectic cohomology.
\newblock In {\em Advances in the theory of automorphic forms and their
  {$L$}-functions}, volume 664 of {\em Contemp. Math.}, pages 321--337. Amer.
  Math. Soc., Providence, RI, 2016.

\bibitem{NS17}
J.~Nekov\'{a}\v{r} and A.~J. Scholl.
\newblock {Plectic Hodge theory I}.
\newblock \url{https://www.dpmms.cam.ac.uk/~ajs1005/preprints/plechdgI.pdf},
  2017.

\bibitem{PR22}
G.~{Pappas} and M.~{Rapoport}.
\newblock {On integral local Shimura varieties}.
\newblock {\em arXiv e-prints}, page arXiv:2204.02829, Feb. 2024, 2204.02829.

\bibitem{PR21}
G.~Pappas and M.~Rapoport.
\newblock {$p$}-adic shtukas and the theory of global and local {S}himura
  varieties.
\newblock {\em Camb. J. Math.}, 12(1):1--164, 2024.

\bibitem{RV14}
M.~Rapoport and E.~Viehmann.
\newblock Towards a theory of local {S}himura varieties.
\newblock {\em M\"{u}nster J. Math.}, 7(1):273--326, 2014.

\bibitem{RZ96}
M.~Rapoport and T.~Zink.
\newblock {\em Period spaces for {$p$}-divisible groups}, volume 141 of {\em
  Annals of Mathematics Studies}.
\newblock Princeton University Press, Princeton, NJ, 1996.

\bibitem{Ryd11}
D.~Rydh.
\newblock Representability of {H}ilbert schemes and {H}ilbert stacks of points.
\newblock {\em Comm. Algebra}, 39(7):2632--2646, 2011.

\bibitem{Sch17}
P.~{Scholze}.
\newblock {Etale cohomology of diamonds}.
\newblock {\em arXiv e-prints}, page arXiv:1709.07343, Jan. 2022, 1709.07343.

\bibitem{SW20}
P.~Scholze and J.~Weinstein.
\newblock {\em Berkeley {L}ectures on $p$-adic {G}eometry}, volume 207 of {\em
  Annals of Mathematics Studies}.
\newblock Princeton University Press, Princeton, NJ, 2020.

\bibitem{SZ17}
X.~Shen and C.~Zhang.
\newblock Stratifications in good reductions of {S}himura varieties of abelian
  type.
\newblock {\em Asian J. Math.}, 26(2):167--226, 2022.

\bibitem{Tam19}
M.~Tamiozzo.
\newblock {\em On the Bloch-Kato conjecture for Hilbert modular forms}.
\newblock PhD thesis, Jul 2019.

\bibitem{Var98}
Y.~Varshavsky.
\newblock {$p$}-adic uniformization of unitary {S}himura varieties. {II}.
\newblock {\em J. Differential Geom.}, 49(1):75--113, 1998.

\bibitem{Vie21}
E.~Viehmann.
\newblock On {N}ewton strata in the {$B_{\rm dR}^+$}-{G}rassmannian.
\newblock {\em Duke Math. J.}, 173(1):177--225, 2024.

\bibitem{YZ17}
Z.~Yun and W.~Zhang.
\newblock Shtukas and the {T}aylor expansion of {$L$}-functions.
\newblock {\em Ann. of Math. (2)}, 186(3):767--911, 2017.

\bibitem{Zhu17b}
X.~Zhu.
\newblock Affine {G}rassmannians and the geometric {S}atake in mixed
  characteristic.
\newblock {\em Ann. of Math. (2)}, 185(2):403--492, 2017.

\end{thebibliography}
\end{document}